\documentclass[a4paper]{amsart}

\usepackage[utf8]{inputenc}
\usepackage{relsize}
\usepackage{amssymb}
\usepackage[hidelinks]{hyperref}
\usepackage[textsize=small]{todonotes}
\usepackage{tikz}
\usetikzlibrary{decorations.text}
\usepackage{xcolor}
\usepackage{enumitem}
\usepackage[normalem]{ulem}
\usepackage{mathrsfs}
\usepackage{graphicx} 

\theoremstyle{plain}
\newtheorem{theorem}{Theorem}[section]
\newtheorem*{theorem-wo-number}{Theorem}
\newtheorem{lemma}[theorem]{Lemma}
\newtheorem{corollary}[theorem]{Corollary}
\newtheorem{proposition}[theorem]{Proposition}

\theoremstyle{definition}
\newtheorem{definition}[theorem]{Definition}

\newtheorem{example}[theorem]{Example}

\newtheorem{question}[theorem]{Question}

\theoremstyle{remark}

\newtheorem*{claim}{Claim}

\newcommand{\bR}{\mathbb{R}}
\newcommand{\R}{\bR}
\newcommand{\bQ}{\mathbb{Q}}
\newcommand{\Q}{\bQ}
\newcommand{\bN}{\mathbb{N}}
\newcommand{\N}{\bN}

\newcommand{\cA}{\mathcal{A}}

\newcommand{\cC}{\mathcal{C}}
\newcommand{\C}{\cC}

\newcommand{\cF}{\mathcal{F}}
\newcommand{\F}{\cF}

\newcommand{\cI}{\mathcal{I}}
\newcommand{\I}{\cI}
\newcommand{\cJ}{\mathcal{J}}
\newcommand{\J}{\cJ}

\newcommand{\cP}{\mathcal{P}}

\newcommand{\continuum}{\mathfrak{c}}

\newcommand{\FIN}{\mathsf{Fin}} 
\newcommand{\fin}{\FIN}
\newcommand{\Fin}{\FIN}

\newcommand{\conv}{\mathsf{conv}} 
\newcommand{\BI}{\mathsf{BI}}

\newcommand{\Hindman}{\mathcal{H}}

\DeclareMathOperator{\FS}{\mathsf{FS}}

\DeclareMathOperator{\FinBW}{FinBW}

\begin{document}


\title{Critical partition regular functions for compact spaces}


\author[R.~Filip\'{o}w]{Rafa\l{} Filip\'{o}w}
\address[Rafa\l{}~Filip\'{o}w]{Institute of Mathematics\\ Faculty of Mathematics, Physics and Informatics\\ University of Gda\'{n}sk\\ ul.~Wita Stwosza 57\\ 80-308 Gda\'{n}sk\\ Poland}
\email{rafal.filipow@ug.edu.pl}
\urladdr{\url{http://mat.ug.edu.pl/~rfilipow}}

\author[M.~Kowalczuk]{Ma\l{}gorzata Kowalczuk}
\address[Ma\l{}gorzata Kowalczuk]{Institute of Mathematics\\ Faculty of Mathematics, Physics and Informatics\\ University of Gda\'{n}sk\\ ul.~Wita Stwosza 57\\ 80-308 Gda\'{n}sk\\ Poland}
\email{m.kowalczuk.090@studms.ug.edu.pl}

\author[H.~Ksi\c{a}\.{z}ek]{Hubert Ksi\c{a}\.{z}ek}
\address[Hubert Ksi\c{a}\.{z}ek]{Institute of Mathematics\\ Faculty of Mathematics, Physics and Informatics\\ University of Gda\'{n}sk\\ ul.~Wita Stwosza 57\\ 80-308 Gda\'{n}sk\\ Poland}
\email{h.ksiazek.596@studms.ug.edu.pl}

\author[A.~Kwela]{Adam Kwela}
\address[Adam Kwela]{Institute of Mathematics\\ Faculty of Mathematics\\ Physics and Informatics\\ University of Gda\'{n}sk\\ ul.~Wita  Stwosza 57\\ 80-308 Gda\'{n}sk\\ Poland}
\email{Adam.Kwela@ug.edu.pl}
\urladdr{\url{https://mat.ug.edu.pl/~akwela}}

\author[G.~Ucal]{Grzegorz Ucal}
\address[Grzegorz Ucal]{Institute of Mathematics\\ Faculty of Mathematics, Physics and Informatics\\ University of Gda\'{n}sk\\ ul.~Wita Stwosza 57\\ 80-308 Gda\'{n}sk\\ Poland}
\email{g.ucal.094@studms.ug.edu.pl}

\thanks{The fourth-listed author was supported by the Polish National Science Centre project OPUS No. 2024/53/B/ST1/02494.}


\date{\today}







\begin{abstract}
We study ideal-based refinements of sequential compactness arising from the class $\FinBW(\I)$, consisting of topological spaces in which every sequence admits a convergent subsequence indexed by a set outside a given ideal $\I$. A central theme of this work is the existence of critical ideals whose position in the Kat\v{e}tov order determines the relationship between a fixed class of spaces and the corresponding $\FinBW(\I)$ classes. Building on earlier results characterizing several classical topological classes via such ideals, we extend this theory to a broader framework based on partition regular functions, which unifies ordinary convergence with other non-classical convergence notions such as IP- and Ramsey-type convergence. Furthermore, we investigate the existence of critical ideals associated with function classes motivated by Mazurkiewicz’s theorem on uniformly convergent subsequences.
\end{abstract}


\maketitle




\section{Introduction}

See Section~\ref{sec:prelim} for notions and notations used in the introduction.

In recent years, various ideal-based refinements of classical convergence and topological properties have been intensively studied. One such notion is the class $\FinBW(\I)$ (associated to an ideal $\I$ on $\omega$) which consists of all topological  spaces $X$ having the property that for every sequence $(x_n)_{n\in \omega}$ in $X$ there exists $A\notin \I$ such that the subsequence $(x_n)_{n\in A}$ is convergent in $X$ (\cite[Definition~1.1]{MR4584767}).
For instance, $\FinBW(\Fin)$ coincides with the class of all sequentially compact  spaces, where $\Fin$ is the ideal of all finite subsets of $\omega$.

A natural  problem in this area is to determine whether a given class of topological spaces can be characterized as a $\FinBW(\I)$ for a suitable ideal $\mathcal I$.

A particularly interesting question concerns the existence of a \emph{critical ideal} $\I_{\C}$ for a given class $\C$ of topological spaces with the following property: 
the position of a given ideal $\J$ relative to $\I_{\C}$ in the Kat\v{e}tov order determines the relationship between the classes $\C$ and $\FinBW(\J)$.
Problems of this type have already been investigated in the literature, see for instance
\cite{Gosia}, \cite{MR4584767} and 
 \cite{alcantara-phd-thesis}.
Some of these results are as follows.
\begin{theorem}\ 
\label{thm:laisurhfkasd}
    \begin{itemize}
        \item $\Fin\otimes \Fin \leq_K\I$ $\iff$ $\FinBW(\I)$ coincides with the class of all finite spaces.
        \item The following conditions are equivalent.
        \begin{enumerate}
            \item $\BI  \leq_K\I$ and $\Fin\otimes\Fin\not\leq_K \I$.
            \item $\FinBW(\I)$ coincides with the class of all boring  spaces.
            \end{enumerate}
    \item 
$\conv\not\leq_K \I$
$\iff$
$\FinBW(\I)$ coincides with the class of all  compact metric spaces in the realm of metric spaces.
\end{itemize}
\end{theorem}

The definition of $\FinBW(\I)$ is based on  the notion of ordinary  convergence. However, there exist other types of convergence that fall outside this classical framework. Examples include convergences arising in ergodic theory and other areas, for instance, \emph{IP-convergence} (introduced by Furstenberg and Weiss \cite{MR531271})
and 
\emph{R-convergence} (considered in  \cite{MR4741322}, \cite{MR2948679} and \cite{MR3019575}).
Such convergences have been employed, for example, in \cite{MR1887003}, \cite{MR1950294} and \cite{MR4552506} to define some classes of topological spaces, namely \emph{Hindman spaces} and \emph{Ramsey spaces}.  
In general, these forms of convergence cannot be adequately described using the language of ordinary convergence alone.

More recently, in the paper \cite{FKK-Unified}, the authors introduced a unifying framework (\emph{partition regular functions}) which, in one formulation, can be interpreted in terms of ordinary convergence, while simultaneously encompassing these other, non-ordinary types of convergence.

Motivated by these developments, the aim of the present paper is to generalize the results of 
Theorem~\ref{thm:laisurhfkasd} for partition regular functions. This task is addressed in Sections~\ref{sec:lkasudhfaks}, \ref{sec:akjsfoias}
and \ref{sec:kajshfdlsa}.

In the final section, we address the problem of identifying the critical ideal for the class of functions possessing the property motivated by Mazurkiewicz’s theorem~\cite{Mazurkiewicz1932}, which states that every  uniformly bounded sequence of continuous functions has a uniformly convergent subsequence once restricted to some perfect set.


\section{Preliminaries}
\label{sec:prelim}

All topological spaces considered in the paper are assumed to be Hausdorff.
 
Recall that an ordinal number $\alpha$ is equal to the set of all ordinal numbers less than $\alpha$. 
In particular, the  smallest infinite ordinal number $\omega=\{0,1,\dots\}$ is equal to the set of all natural numbers $\N$, and each natural number $n = \{0,\dots,n-1\}$ is equal  to the set of all natural numbers less than $n$.
Using this identification, we can, for instance, write $n\in k$ instead of $n<k$, $n<\omega$ instead of $n\in \omega$, or $A\cap n$ instead of $A\cap \{0,1,\dots,n-1\}$. 

 We write  
$[A]^{<\omega}$ to denote the family of all finite subsets of $A$,
$[A]^\omega$ to denote the family of all infinite countable subsets of $A$
and
$\cP(A)$ to denote the family of all subsets of $A$.
We say that  $\cA\subseteq [\omega]^\omega$
is an \emph{almost
disjoint family on $\omega$} if $A\cap B$ is finite for any distinct $A,B\in \cA$.


\subsection{Ideals and partition regular functions}

An \emph{ideal} on a nonempty set $X$ is a  family $\I\subseteq\cP(X)$ that satisfies the following properties:
\begin{enumerate}
\item $\emptyset\in \I$ and $X\not\in\I$;
\item if $A,B\in \I$ then $A\cup B\in\I$;
\item if $A\subseteq B$ and $B\in\I$ then $A\in\I$;
\item $\I$ contains all finite subsets of $X$. 
\end{enumerate}
For an ideal $\I$,  we write $\I^+=\{A\subseteq X: A\notin\I\}$ and call it the \emph{coideal of $\I$}.
Sets belonging to $\I^+$ will be called \emph{$\I$-positive} sets.
An ideal $\I$ is \emph{tall} (also called \emph{dense}) if for every infinite $A\subseteq X$ there is an infinite $B\in\I$ such that $B\subseteq A$.

\begin{example}\  
\begin{enumerate}
    \item 
     $\fin$ is the ideal of all finite subsets of $\omega$. 

\item 
$\Fin^2= \Fin\otimes\Fin$ is the ideal on $\omega^2=\omega\times \omega$ defined by
\begin{equation*}
A\in \Fin^2 \iff \exists i_0\in\omega\, \forall i\geq i_0\, (|\{j\in\omega: (i,j)\in A\}|<\omega).
\end{equation*}

\item $\BI$ is the ideal  on $\omega^3 = \omega\times \omega \times\omega$
introduced in \cite[Definition~4.1]{MR4584767}
and defined by 
\begin{equation*}
    \begin{split}
A\in \BI 
\iff 
\exists i_0\in\omega\, \left[\right.
&
\forall i< i_0\, (\, \{(j,k)\in\omega^2 : (i,j,k)\in A\}\in \Fin^2)
\  \land 
\\&
\left.
\hspace{-3pt}\forall i\geq i_0\, (|\{(j,k)\in\omega^2 : (i,j,k)\in A\}|<\omega)
\right].
    \end{split}
\end{equation*}
The ideal $\BI$ is isomorphic to the ideal $\mathcal{WT}$ introduced in \cite{BRENDLE-CJM-2025}, where the authors showed its connection with weakly tight almost disjoint families.

\item 
$\conv$ is the ideal on $\Q\cap [0,1]$ 
consisting of all subsets of $\Q\cap [0,1]$ which can be covered by ranges of finitely many sequences in $\Q\cap [0,1]$ which are convergent in $[0,1]$.

\end{enumerate}\end{example}

\begin{definition}[\cite{FKK-Unified}]
\label{def:partition-regular}
Let $\Lambda$ and $\Omega$ be  countable infinite sets.
Let $\cF$ be a nonempty family of infinite subsets of $\Omega$ such that $F\setminus K\in \cF$ for every $F\in \cF$ and every finite set $K\subseteq \Omega$.
We say that a function $\rho\colon\cF\to [\Lambda]^\omega$ 
is \emph{partition regular} if it satisfies the following conditions:
\begin{description} 
\item[(M)]
$\forall E,F\in \cF\, \left( E\subseteq F \implies \rho(E)\subseteq \rho(F)\right)$;

\item[(R)]
$\forall F\in \cF\,\forall A,B \subseteq \Lambda \,(\rho(F) =  A\cup B  \implies  
\exists E\in \cF \,(\rho(E)\subseteq A \lor \rho(E)\subseteq B));$

\item[(S)]
$\forall E\in \cF\,\exists F\in \cF\,(F\subseteq E \land \forall a\in \rho(F)\,\exists K\in[\Omega]^{<\omega}\, (a\notin \rho(F\setminus K)))$.

\end{description}
\end{definition}

The following proposition reveals  relationships between partition
regular functions and ideals.

\begin{proposition}[{\cite[Proposition~3.3]{FKK-Unified}}]\
\label{prop:rho-versus-ideal}
\begin{enumerate}
\item If $\rho\colon\cF\to[\Lambda]^\omega$ is partition regular then 
$\I_{\rho} = \{A\subseteq \Lambda: \forall  F\in \cF\, (\rho(F)\not\subseteq A)\}$
is an ideal on $\Lambda$. 
\label{prop:rho-versus-ideal:rho-gives-ideal}

    \item 
If $\I$ is an ideal on $\Lambda$ then 
the function 
$\rho_{\I}\colon\I^+\to[\Lambda]^\omega$ given by $\rho_{\I}(A)=A$
is partition regular and  $\I=\I_{\rho_{\I}}.$ \label{prop:rho-versus-ideal:ideal-gives-rho} 
\end{enumerate}
\end{proposition}

The above proposition shows that every ideal can be coded using partition regular functions, namely, for an ideal $\I$ we have a partition regular function $\rho = \rho_\I$ with $\I = \I_{\rho}$. On the other hand, not every partition regular function can be coded using ideals.
For instance,  $\FS\neq\rho_\I$ for any ideal $\I$ (\cite[Propositions~6.5 and 6.7]{FKK-Unified}), where 
the partition regular function 
$\FS: [\omega]^\omega\to [\omega]^\omega$ is given by 
\begin{equation*}
\FS(D)  = \{ d_0+\dots+d_n : n\in\omega,d_0,\dots,d_n\in D, d_i\neq d_j  \text{for } i\neq j\},
\end{equation*}
 and the ideal  
 \begin{equation*}
\Hindman = \I_{\FS} = \{A\subseteq \omega: \forall  D\in [\omega]^\omega\, (\FS(D)\not\subseteq A)\}
\end{equation*}
 is called \emph{the Hindman ideal} \cite[p.~109]{MR2471564} (see also \cite{MR1887003}).

A set $F\in \cF$ has \emph{small accretions} if 
$\rho(F)\setminus\rho(F\setminus K)\in\I_\rho$ for every finite set $K$. We say that  $\rho$  has \emph{small accretions} if for every $E\in \cF$, there is $F\in \cF$ with  $F\subseteq E$ such that  $F$ has small accretions.


\subsection{$P$-like properties}

In this section $\I$ is an ideal on~$\omega$, and 
$\rho\colon \cF\to[\Lambda]^\omega$ with $\cF\subseteq[\Omega]^\omega$ is  a  partition regular function.

We write:
\begin{itemize}
\item $\rho \in P^-(\Lambda) $ if for every decreasing sequence  $\Lambda=A_0\supseteq A_1\supseteq\dots $ such that 
$A_n\setminus A_{n+1}\in \I_\rho$ for each $n$  there exists 
$F \in \mathcal{F}$ such that for each $n\in \omega$ there exists a finite set $K$ with 
$\rho(F \setminus K) \subseteq A_n$.
\item $\rho \in P^- $ if for every decreasing sequence  $A_0\supseteq A_1\supseteq\dots $ such that $A_0\in \I_\rho^+$ and  
$A_n\setminus A_{n+1}\in \I_\rho$ for each $n$  there exists 
$F \in \mathcal{F}$ such that
$\rho(F)\subseteq A_0$ and 
for each $n\in \omega$ there exists a finite set $K$ with 
$\rho(F \setminus K) \subseteq A_n$.
\end{itemize}

If $\rho_\I\in P^-(\omega)$ ($\rho_\I\in P^-$) then we write $\I \in P^-(\omega)$ ($\I\in P^-$, respectively). Thus, equivalently, $\I \in P^-(\omega)$ if for every decreasing sequence  $\omega=A_0\supseteq A_1\supseteq\dots $ such that 
$A_n\setminus A_{n+1}\in \I_\rho$ for each $n$ 
there exists  $B\in \I^+$ such  that for each $n\in \omega$ there exists a finite set $K$ with 
$B\setminus K \subseteq A_n$. Similarly one can express $\I\in P^-$.

It is not difficult to show the following characterizations of $P^-$-like properties which we will sometimes used in the paper.

\begin{proposition}
    \ 
    \begin{enumerate}
\item 
$\I \in P^-(\omega)$ $\iff$ for every partition $\{A_n:n\in \omega\}$ of $\omega$ such that 
$A_n\in \I$ for each $n$, 
there exists  $B\in \I^+$ such that $B\cap A_n$ is finite for each $n\in \omega$.

    \item 
 $\rho \in P^-(\Lambda) $ $\iff$
for every  partition $\{A_n:n\in \omega\}$ of $\Lambda$ such that 
$A_n\in \I_\rho$ for each $n$, there exists 
$F \in \mathcal{F}$ such that for each $n\in \omega$ there exists a finite set $K$ with 
$\rho(F \setminus K) \cap A_n=\emptyset$.

\item 
$\I\in P^-$ $\iff$
for every  family  $\{A_n:n\in \omega\}$ of pairwise disjoint subsets of  $\omega$ such that $\bigcup \{A_n:n\in \omega\}\notin\I$ and 
$A_n\in \I$ for each $n$, there exists $B\subseteq \bigcup \{A_n:n\in \omega\}$  such that  
$B\in \I^+$
and 
$B \cap A_n$  is finite for each $n\in \omega$.

        \item 
$\rho\in P^-$ $\iff$
for every  family  $\{A_n:n\in \omega\}$ of pairwise disjoint subsets of  $\Lambda$ such that $\bigcup \{A_n:n\in \omega\}\notin\I_\rho$ and 
$A_n\in \I_\rho$ for each $n$, there exists 
$F \in \mathcal{F}$ such that for each $n\in \omega$ there exists a finite set $K$ with 
$\rho(F \setminus K) \cap A_n=\emptyset$.

\end{enumerate}
\end{proposition}


\subsection{Kat\v{e}tov order}

Let $\I_1$ and $\I_2$ be ideals on $X$ and $Y$, respectively.
We say that \emph{$\I_1$ is above $\I_2$ in the Kat\v{e}tov order} (in short: $\I_2\leq_K \I_1$)  if 
 there exists  a function $f\colon X\to Y$ 
 such that $f[A]\in \I_2^+$ for each $A\in \I_1^+$.

Let $\rho_i\colon\cF_i\to[\Lambda_i]^\omega$ 
be partition regular with $\cF_i\subseteq [\Omega_i]^\omega$ for each $i=1,2$.
We say that \emph{$\rho_1$ is above  $\rho_2$ in the Kat\v{e}tov order}
(in short: $\rho_2\leq_K \rho_1$)
if  there is a function $f\colon \Lambda_1\to\Lambda_2$
such that 
\begin{equation*}
\forall F_1\in \cF_1 \, \exists  F_2\in \cF_2 \, \forall K_1\in [\Omega_1]^{<\omega}\, \exists K_2\in [\Omega_2]^{<\omega}\,(\rho_2(F_2\setminus K_2)\subseteq f[\rho_1(F_1\setminus K_1)]).
\end{equation*}

We write $\rho_2 \approx_K \rho_1$ if both $\rho_2 \leq_K \rho_1$ and $\rho_1 \leq_K \rho_2$ hold.


\subsection{Stronger sparseness of partition regular functions}

In this section, we present a  strengthening of the property (S) from Definition~\ref{def:partition-regular} which will be required in some results of this paper.

\begin{definition}
\label{def:eS-one}
\label{def:eS-two}
Let 
$\rho\colon \cF\to[\Lambda]^\omega$ with $\cF\subseteq[\Omega]^\omega$ be a  partition regular function.
We say that 
\begin{enumerate}
    \item 
 $\rho \in (S_1) $ if 
 for every $E \in \mathcal{F}$ there exists $F \in \mathcal{F}$ with $F\subseteq E$ such that 
for every  $A \in \I_{\rho}$ there exists  $G \in \mathcal{F}$ with $\rho(G) \subseteq \rho (F) \setminus A$ such that 
\begin{equation*}
\forall K\in [\Omega]^{<\omega} \, \exists L\in [\Omega]^{<\omega} \, (\rho(G\setminus L) \subseteq \rho(F \setminus K));
\end{equation*}

    \item 
 $\rho \in (S_2) $ if 
 for every $E \in \mathcal{F}$ there exists $F \in \mathcal{F}$ with $F\subseteq E$ such that 
 for every $B \notin \I_{\rho}$ with $B \subseteq \rho(F)$ 
 there exists $G \in \mathcal{F}$ with $\rho(G) \subseteq B$
 such that 
\begin{equation*}
\forall K\in [\Omega]^{<\omega} \, \exists L \in [\Omega]^{<\omega} \, (\rho (G\setminus L) \subseteq \rho(F \setminus K)).
\end{equation*}

\end{enumerate}
\end{definition}

\begin{proposition} 
Let $\rho\colon  \mathcal{F} \to [\Lambda]^{\omega}$ with $\cF\subseteq [\Omega]^\omega$ be a partition regular function. 
\begin{enumerate}

\item 
 If $\rho \in (S_2)$ then $ \rho \in (S_1)$.

\item  $\rho_{\I} \in (S_2)$ for every ideal $\I$.
\item  If $\rho \in P^-$ and has small accretions, then $ \rho \in (S_2)$.

\end{enumerate}
\end{proposition}

\begin{proof}
(1)
Take any $E \in \mathcal{F}$. 
Let  $F \in \mathcal{F}$ be a witness for $\rho\in (S_2)$.
 If $A\in \I_{\rho}$ then $B=\rho(F)\setminus A \notin \I_{\rho} $, so the set $G$ from the definition of $(S_2)$ works for $(S_1)$.

    (2) Take any $E\in \I^+$, put $F=E$ and fix any $\I$-positive subset $B$ of $E$. Then $G=B$ works, as $B \setminus K\in \I^+$ and $B \setminus K \subseteq E\setminus K$ for all finite sets $K$.

    (3) 
    Enumerate the set $\Omega=\{k_i :i \in \omega\}$ and define $ K_n=\{k_i : i<n \} $. Take any $E \in \mathcal{F}$. Since $\rho$ has small accretions, there exists  some $F\in \mathcal{F}, F \subseteq E$ such that $\rho(F)\setminus \rho(F \setminus K_n) \in \I_{\rho}$ for every $n\in \omega$. Take $B \notin \I_{\rho}$ with $B\subseteq\rho(F)$ and define  $A_n=\rho(F \setminus K_n) \cap B $ for each $n\in \omega$. Then, $(A_n)_{n\in \omega}$ is a decreasing sequence such that $A_n\setminus A_{n+1} \in \I_{\rho}$ and $A_0\notin\I_\rho$. Since $\rho \in P^-$, there exists a set $G\in \mathcal{F} $ such that $\rho(G)\subseteq A_0$ and for every $n \in \omega$ there exists a finite set $M_n$ with $\rho(G \setminus M_n) \subseteq A_n$. Then,  $\rho(G) \subseteq B$ and for every $i\in \omega$, we get that $\rho(G\setminus M_i)\subseteq \rho(F \setminus K_i)$. 
\end{proof}


\subsection{Some classes of topological spaces}

A sequentially compact space $X$ is called \emph{boring} (\cite[Definition~3.21]{MR4584767}) if there exists a finite set $F\subseteq X$ such that each one-to-one sequence in $X$ converges to some point of  $F$. In case of infinite Hausdorff spaces, $X$ is boring if and only if it is a disjoint union of finitely many one-point compactifications
of discrete spaces.

For an ideal $\I$ on $\omega$, we write $\FinBW(\I)$ to denote the class of all topological spaces~$X$ such that for every sequence $f\colon \omega\to X$ there exists $A\in \I^+$ such that $f\restriction A$ converges. 
We say that an ideal $\I$ has the \emph{$\FinBW$ property} if $[0,1]\in \FinBW(\I)$.

For a partition regular function $\rho\colon \cF\to[\Lambda]^\omega$, 
we define  $\FinBW(\rho)$ to be the class of all topological spaces $X$ such that for every sequence $f \colon \Lambda \to X$ there exists $F \in \mathcal{F}$ such that $f \restriction \rho(F)$ \emph{$\rho$-converges} i.e. there exists $x\in X$ such that 
for every neighborhood $U$ of $x$ there is a finite set $K$ with 
$f[\rho(F\setminus K)] \subseteq U$.


\section{Critical \texorpdfstring{$\rho$}{rho} for the class of all finite spaces}
\label{sec:lkasudhfaks}

\begin{theorem}
\label{thm:Finite-Spaces-via-FinSqrd}
Let $\rho : \mathcal{F} \to [\Lambda]^{\omega}$ be partition regular
with $\cF\subseteq[\Omega]^\omega$.
If  $\rho \in (S_1)$, then the following conditions are equivalent:
    \begin{enumerate}
        \item  $\FinBW(\rho)$ coincides with the class of all finite spaces,\label{thm:Finite-Spaces-via-FinSqrd:item-FinBW}

\item $\omega+1\notin \FinBW(\rho)$, where the ordinal $\omega+1$ is considered as a topological space with  the order topology,\label{thm:Finite-Spaces-via-FinSqrd:item-omega-plus-one}

\item $\rho \notin P^-(\Lambda)$,\label{thm:Finite-Spaces-via-FinSqrd:item-Pminus} 
        
        \item $\rho_{\Fin^2} \leq_K \rho$.\label{thm:Finite-Spaces-via-FinSqrd:item-rho-FinSqrd}
        \end{enumerate}
The assumption  $\rho\in (S_1)$  is only used to prove the implication $(\ref{thm:Finite-Spaces-via-FinSqrd:item-Pminus}) \implies (\ref{thm:Finite-Spaces-via-FinSqrd:item-rho-FinSqrd})$.  
\end{theorem}

\begin{proof} 
$(\ref{thm:Finite-Spaces-via-FinSqrd:item-FinBW})\iff(\ref{thm:Finite-Spaces-via-FinSqrd:item-Pminus})$ This is proved in \cite[Theorem~10.5(3)]{FKK-Unified}.


$(\ref{thm:Finite-Spaces-via-FinSqrd:item-FinBW})\implies(\ref{thm:Finite-Spaces-via-FinSqrd:item-omega-plus-one})$ This is obvious.


$(\ref{thm:Finite-Spaces-via-FinSqrd:item-omega-plus-one})\implies(\ref{thm:Finite-Spaces-via-FinSqrd:item-Pminus})$ 
Let $f:\Lambda\to\omega+1$ be a witness for $\omega+1\notin \FinBW(\rho)$.
We define $A_\alpha=f^{-1}(\alpha)$ for each $\alpha\in \omega+1$.
Once, we show that  the family $\{A_\alpha:\alpha\in \omega+1\}$ is a witness for $\rho\notin\ P^-(\Lambda)$, the proof will be finished. Observe that $A_\alpha\in\I_\rho$ for all $\alpha\in \omega+1$ as otherwise we could find $F\in\cF$ and $\alpha\in \omega+1$ with $\rho(F)\subseteq A_\alpha$ such that $f\restriction \rho(F)$ would be constant, so $\rho$-convergent. Suppose for the sake of contradiction that there is $F\in \cF$ such that for each $\alpha$ there is a finite set $K_\alpha$ with $\rho(F\setminus K_\alpha)\cap A_\alpha=\emptyset$.
We claim that $f\restriction \rho(F)$ is $\rho$-convergent to the point $\omega$ (and that will finish the proof). 
Let $U$ be a neighborhood of the point $\omega$. Then, there is $n\in \omega$ with $(n,\omega]\subseteq U$.
Then,  $K = \bigcup_{\alpha\leq n}K_\alpha$ is finite and $\rho(F\setminus K)\cap A_\alpha=\emptyset$ for each $\alpha\leq n$, so $f[\rho(F\setminus K)]\subseteq U$.


    $(\ref{thm:Finite-Spaces-via-FinSqrd:item-Pminus}) \implies (\ref{thm:Finite-Spaces-via-FinSqrd:item-rho-FinSqrd})$ 
Since $\rho \notin P^-(\Lambda)$,  there exists a partition $\{A_n:n\in \omega\}$ of $\Lambda$ such that $A_n \in \I_{\rho}$ for each $n \in \omega $ and 
for every $F\in\mathcal{F}$ there exists $n \in \omega$
such that $\rho(F \setminus L) \cap A_n \neq \emptyset$
for any finite set $L$.
Let $f: \Lambda \to \omega^2$ be an injective function such that $f[A_n] \subseteq \{n\}\times\omega$ for each $n$. 
Once we show that  $f$ is a  witness for 
$\rho_{\Fin^2} \leq_K \rho$, the proof will be finished.

   Take any $F \in \mathcal{F}$. 
   To finish the proof, we need to find $B\notin \Fin^2$ such that for each finite $K$ there is a finite $L$ with 
   $B\setminus L\subseteq f[\rho(F\setminus K)]$.

   Since $\rho \in (S_1)$,  there exists $E\in \mathcal{F}$ with $E\subseteq F$ 
   such that 
for every $A\in \I_{\rho}$ there exists  $H_A \in \mathcal{F}$ with $\rho(H_A) \subseteq \rho(E) \setminus A$ such that 
\begin{equation*}
\forall K\in [\Omega]^{<\omega} \ \exists L_{K,A}\in [\Omega]^{<\omega} \, (\rho(H_A\setminus L_{K,A}) \subseteq \rho (E\setminus K)).
\end{equation*}

Let $\Omega=\{k_i :i \in \omega\}$ and  $ K_n=\{k_i : i<n \} $ for each $n\in \omega$. 
We will construct $G_l \in \mathcal{F}$, $\{a_n^l:n\in \omega\}\subseteq  \Lambda$ and $ n_l \in \omega $ by the induction on $l \in \omega$ in such a way that the following conditions are satisfied:
\begin{enumerate}
\item  $G_l \subseteq H_{\bigcup_{i<l}A_{n_i}}$ for each $l$, 
\item  $a_n^l \in \rho(G_l \setminus L_{K_{l+n},\bigcup_{i<l}A_{n_i}} ) \cap A_{n_l}$ for each $l$ and $n$, 
\item  $a_n^l \neq a_k^l$ for each $l$ and distinct $n,k$,
\item  $n_l \neq n_k$ for distinct $l,k$.  
\end{enumerate}

Assume $G_j \in \mathcal{F}$, $\{a_n^j:n\in \omega\}\subseteq  \Lambda$ and $n_j \in \omega $ have already been constructed for $j<l$.
Since  $H_{\bigcup_{i<l}A_{n_i}} \in \mathcal{F}$, we can find $G_l \in \mathcal{F}$
such that 
$G_l \subseteq H_{\bigcup_{i<l}A_{n_i}}$
and for every $a\in \rho(G_l)$ there is a finite set $K$ with $a\notin \rho(G_l\setminus K)$.
Let $n_l$ be such that  $\rho(G_l\setminus L) \cap A_{n_l} \neq \emptyset$ for every finite $L$. Note that $n_l\neq n_i$ for all $i<l$ as $\rho(G_l\setminus L)\subseteq\rho(G_l)\subseteq\rho(H_{\bigcup_{i<l}A_{n_i}})\subseteq \rho(E)\setminus \bigcup_{i<l}A_{n_i}$.
Now, we will inductively pick elements $a^l_n$ for $n\in \omega$.
Suppose we have already chosen $a_i^l$ for $i<n$. 
 For each $i<n$, we can find a finite set  $M_i^l$ such that 
$a_i^l \notin \rho(G_l \setminus (L_{K_{l+n}, \bigcup_{i<l}A_{n_i}} \cup M_i^l)$. 
Now we pick  any 
\begin{equation*}
a_n^l \in \rho\left(G_l \setminus \left( L_{K_{l+n}, \bigcup_{i<l}A_{n_i}} \cup \bigcup_{i<n} M_i^l\right)\right) \cap A_{n_l}.
\end{equation*}
 That ends the inductive construction. 

Let $B=\{f(a_{n}^l) :l,n \in \omega \}$. 
Since $f$ is injective and $f(a^l_n)\in\{n_l\}\times \omega$ for each $l,n\in \omega$ (by $a^l_n\in A_{n_l}$), we obtain that 
$B\cap (\{n_l\} \times \omega)$ is infinite for  each $l \in \omega$.  Thus, $B\notin \Fin^2$.
Let's take any finite set $K$.  
Let $j_0$ be such that $K\subseteq K_{j_0}$.
Let $L=\{f(a_n^j) : n,j<j_0 \}$.
The set  $L$ is finite, so the proof will be finished once we show that  $B\setminus L\subseteq f[\rho(F\setminus K)]$.
Take any $f(a_n^l)\in B\setminus L$.
Then, $l\geq j_0$ or $n\geq j_0$, so $l+n\geq j_0$, and consequently we have 
\begin{equation*}
\begin{split}   
a^l_n 
&\in 
\rho(G_l \setminus L_{K_{l+n},\bigcup_{i<l}A_{n_i}}) 
\subseteq  
\rho(H_{\bigcup_{i<l}A_{n_i}} \setminus L_{K_{l+n},\bigcup_{i<l}A_{n_i}}) 
\\&\subseteq 
\rho(E\setminus K_{l+n} ) 
\subseteq
\rho(E\setminus K_{j_0} ) 
\subseteq 
\rho(F\setminus K), 
\end{split}    
\end{equation*}
so $f(a_n^l)\in f[ \rho(F\setminus K)]$.


$(\ref{thm:Finite-Spaces-via-FinSqrd:item-rho-FinSqrd})\implies (\ref{thm:Finite-Spaces-via-FinSqrd:item-FinBW})$ 
It is easy to see that each finite space is $\FinBW(\rho)$ for every partition regular function $\rho$. If $\rho_{\Fin^2} \leq_K \rho $, then $\FinBW(\rho) \subseteq \FinBW(\rho_{\Fin^2})  =\FinBW(\Fin^2)$
by \cite[Theorem~11.1 and Proposition~10.2]{FKK-Unified}, and $\FinBW(\Fin^2)$  coincides with the class of all finite spaces by \cite[Theorem~6.5]{MR4584767}. 
\end{proof}

\begin{corollary}[{\cite[Theorem~6.5 and Proposition~2.5]{MR4584767}}]
\label{cor:ksahdflkhsadflksad}
    Let  $\I$ be an ideal on $\omega$. The following conditions are equivalent.
    \begin{enumerate}
        \item  $\FinBW(\I)$  coincides with the class of all finite spaces.
        \item $\I \not\in P^-(\omega).$
        \item $\Fin^2 \leq_K \I.$
        \end{enumerate}
\end{corollary}

\begin{proof}
Take $\rho=\rho_\I$ and apply  Theorem~\ref{thm:Finite-Spaces-via-FinSqrd}
and \cite[Propositions~6.5(2), 7.5(2b) and 10.2(4)]{FKK-Unified}.
\end{proof}

\begin{question}
Is there a partition regular function 
$\eta$ such that the following conditions are equivalent 
\begin{enumerate}
    \item 
$\FinBW(\rho)$ coincides with the class of all finite spaces,
\item  $\eta \leq_K \rho$
\end{enumerate}
for every partition regular function $\rho$?
\end{question}


\section{Critical \texorpdfstring{$\rho$}{rho} for the class of all boring spaces}
\label{sec:akjsfoias}

First, we introduce  a  strengthening of the property $P^-(\Lambda)$ which will be used  to characterize the class of all boring spaces.

\begin{definition}
Let $\rho : \mathcal{F} \to [\Lambda]^{\omega}$ be partition regular with $\cF\subseteq[\Omega]^\omega$.
We say that  
\begin{enumerate}
    \item 
$\rho \in P_1^-(\Lambda) $ if $\rho\in P^-(\Lambda)$;

\item 
$\rho\in P_2^-(\Lambda)$ if 
for every partition $\{A_{i,j}: i,j\in \omega\}$ of $\Lambda$ such that $A_{i,j} \in \I_{\rho}$ for each $i,j\in \omega$ there exists  $F \in \mathcal{F} $ such that
    \begin{enumerate}
    \item $ \forall i,j \, \exists K\in [\Omega]^{<\omega} \, (\rho(F\setminus K)\cap A_{i,j}=\emptyset) $,
    \item $\exists i_0\,\forall i>i_0\, \exists K\in [\Omega]^{<\omega} \, (\rho(F\setminus K)\cap \bigcup_{j<\omega} A_{i,j}=\emptyset) $.
    \end{enumerate}
\end{enumerate}
It is not difficult to notice that $\rho \in P_2^-(\Lambda)$ implies $\rho \in P_1^-(\Lambda)$.
\end{definition}

\begin{theorem}
\label{thm:rho-BI-vs-P2minus-vs-FinBW}
Let $\rho : \mathcal{F} \to [\Lambda]^{\omega}$ be partition regular with $\mathcal{F} \subseteq [\Omega]^{\omega}$. 
\begin{enumerate}
    \item 
    If $\rho \in (S_1)$ then 
\begin{equation*}
\rho_{\BI}\leq_K \rho \iff \rho \not\in P^-_2(\Lambda).
\end{equation*}
The assumption  $\rho\in (S_1)$  is only used to prove the implication ``$\impliedby$''.\label{thm:rho-BI-vs-P2minus-vs-FinBW:item-rho-BI-vs-P2minus}

\item 
If $\rho \in (S_2)$, then 
\begin{equation*}
\rho \in P_2^-(\Lambda) \iff \omega^2 +1 \in \FinBW(\rho),
\end{equation*}
     where the ordinal $\omega^2+1$ is considered as a topological space with  the order topology.
The assumption  $\rho\in (S_2)$  is only used to prove the implication ``$\implies$''.\label{thm:rho-BI-vs-P2minus-vs-FinBW:item-P2minus-vs-FinBW} 
\end{enumerate}
\end{theorem}



\begin{proof}
    (\ref{thm:rho-BI-vs-P2minus-vs-FinBW:item-rho-BI-vs-P2minus}, $\implies$)  Let $f:\Lambda \to \omega^3$
    be a witness for $\rho_{BI}\leq_K \rho$.    
We define  $A_{i,j}=f^{-1}[\{i\} \times \{j\} \times \omega]$ for each $i,j\in \omega$.
Once we show that $\{A_{i,j}:i,j\in \omega\}$ is a witness for  $\rho \notin P^-_2(\Lambda)$, the proof will be finished.

First, we observe that $\{A_{i,j}: i,j\in \omega\}$ is a partition of $\Lambda$ and  $A_{i,j} \in \I_{\rho}$ because the function $f$ is a witness for $\rho_{BI}\leq_K \rho $. 
Now, we take any $ F \in \mathcal{F}$.
Since $f$ is a witness for $\rho_{BI}\leq_K \rho$, there  exists $G\notin \BI$ 
such that 
for every finite $K$ there is a finite $L$ with $G\setminus L \subseteq f[\rho(F\setminus K)]$.
Since $G\notin \BI$, we have two cases.

\emph{Case (1).} 
$\{(j,k) : (i_0,j,k)\in G\}\notin \Fin^2$ for some $i_0\in \omega$.

Then there exists  $j_0$ such that $D = \{k : (i_0,j_0,k)\in G\}\notin \Fin$.
Once we show that $\rho(F\setminus K)\cap A_{i_0,j_0}\neq\emptyset$
for any finite $K$, this case will be finished.
Take any finite $K$. Then, we can find a finite $L$ with 
$G\setminus L \subseteq f[\rho(F\setminus K)]$.
Since $D$ is infinite, we can find  $k_0$ such that $(i_0,j_0,k_0)\in G\setminus L$.
Then, 
$(i_0,j_0,k_0)\in f[\rho(F\setminus K)]$, so 
$
A_{i_0,j_0}\cap \rho(F\setminus K)
\supseteq 
f^{-1}[\{(i_0,j_0,k_0)\}]\cap \rho(F\setminus K) 
\neq \emptyset$.

\emph{Case (2).}  
$|\{(j,k) : (i,j,k)\in G\}|=\omega$ for infinitely many $i\in \omega$. 

Once we show that for every $i$ there is $i_0>i$ such that 
$\rho(F\setminus K)\cap \bigcup_{j<\omega} A_{i_0,j}\neq \emptyset $ for every finite $K$, this case will be finished. 
Take any $i\in \omega$.
Then, there is $i_0>i$ such that 
$D= \{(j,k) : (i_0,j,k)\in G\}$ is infinite.
Take any finite $K$. 
Then, we can find a finite $L$ with 
$G\setminus L \subseteq f[\rho(F\setminus K)]$.
Since  $\{i_0\}\times D$ is infinite, we can find $(j_0,k_0)\in D$ with 
$(i_0,j_0,k_0)\in G\setminus L$.
Then, 
$(i_0,j_0,k_0)\in f[\rho(F\setminus K)]$, so 
$
\bigcup_{j<\omega} A_{i_0,j}\cap \rho(F\setminus K) 
\supseteq 
f^{-1}[\{(i_0,j_0,k_0)\}]\cap \rho(F\setminus K) 
\neq \emptyset$.


       (\ref{thm:rho-BI-vs-P2minus-vs-FinBW:item-rho-BI-vs-P2minus}, $\impliedby$) 
    Let $\{A_{i,j}: i,j\in \omega\}$ be a witness for $\rho \notin P^-_2(\Lambda)$.     
Let $f:\Lambda \to \omega^3$ be an injective function such that $f[A_{i,j}] \subseteq \{(i,j)\} \times \omega$. Once we show that $f$ is a  witness for $\rho_{\BI}\leq_K \rho$, the proof will be finished.
Take any $F \in \mathcal{F}$.
We will find $G\notin\BI$ such that for every finite set $K$ there exists a finite set $L$ with $G\setminus L\subseteq f[\rho(F\setminus K)]$. 
Let  $E\in \cF$ be such that  $E\subseteq F$ and  for every $a\in \rho(E)$
there exists a finite set $K_a$ with $a\notin \rho(E\setminus K_a)$.
Since  $\{A_{i,j}: i,j\in \omega\}$ is  a witness for $\rho \notin P^-_2(\Lambda)$, we have the following cases.     

\emph{Case (1).} 
For every $i\in \omega$ there is $i_0>i$ such that 
$\rho(E\setminus K)\cap \bigcup_{j<\omega} A_{i_0,j}\neq \emptyset$ for every finite $K$. 

In this case, we can find  a strictly increasing sequence  $(i_n)$  such that for every $n\in \omega$ and every finite $K$
 we have 
\begin{equation*}
\rho(E\setminus K)\cap \bigcup_{j<\omega} A_{i_n,j}\neq \emptyset.
\end{equation*}
Let $g:\omega\to\omega$ be any function such that $g^{-1}(\{n\})$ is infinite for each $n\in \omega$.
Let $\Omega=\{m_i :i \in \omega\}$ and  $ M_n=\{m_i : i<n \} $ for each $n\in \omega$. 
Now, for each $n\in \omega$ we pick any 
\begin{equation*}
a_n \in \rho\left(E\setminus \left(\bigcup_{k<n} K_{a_k} \cup M_n\right)\right)\cap \bigcup_{j<\omega} A_{i_{g(n)},j}.
\end{equation*}
Let 
$G = \{f(a_n):n\in \omega\}$. 
Since $a_n$ are pairwise distinct and $f$ is injective,
we obtain that 
$G\cap (\{i_{n}\}\times\omega\times \omega)$ is infinite for each $n\in \omega$. Consequently, $G\notin \BI$.
Take any finite set $K$ and let $i$ be such that $K\subseteq M_i$.
Let $L=\{f(a_n):n<i\}$.
Then, $L$ is finite and 
for every $f(a_n)\in G\setminus L$ we have
\begin{equation*}
f(a_n)\in  f[\rho(E\setminus M_i)] \subseteq f[\rho(F\setminus K)],
\end{equation*}
so $G\setminus L \subseteq f[\rho(F\setminus K)]$.


\emph{Case (2).}
Case (1) does not hold, and consequently there exist $i,j$ such that $\rho(E\setminus K)\cap A_{i,j}\neq\emptyset$
 for every finite set $K$.

\begin{claim}
There exists an infinite set $\{(i_n,j_n):n\in \omega\}$ of pairs of integers such that 
$\rho(E\setminus K)\cap A_{i_n,j_n}\neq\emptyset$
for every $n$ and every finite set $K$.    
\end{claim}

\begin{proof}[Proof of claim]
    Since $\rho\in (S_1)$, 
there    exists $H \in \mathcal{F}$ with $H\subseteq E$ such that 
for every  $A \in \I_{\rho}$ there exists  $P \in \mathcal{F}$ with $\rho(P) \subseteq \rho(H) \setminus A$ such that
for every finite set $K$ there exists a finite set $L$ with $\rho(P\setminus L)\subseteq \rho(H\setminus K)$.
Since $H\subseteq E$, it is enough to show the claim for the set $H$ instead of $E$.

Now, we construct $i_n,j_n\in \omega$ and $P_n\in \cF$ such that for every $n$
\begin{enumerate}
    \item[(i)] $\rho(H\setminus K)\cap A_{i_0,j_0}\neq \emptyset$ for every finite set $K$,
    
    \item[(ii)] $\rho(P_n)\subseteq\rho(H)\setminus \bigcup\{A_{i_l,j_l}:l\leq n\}$,

    \item[(iii)] for every finite set $K$ there exists a finite set $L$ with $\rho(P_n\setminus L)\subseteq \rho(H\setminus K)$,

    \item[(iv)] $\rho(P_n\setminus K)\cap A_{i_{n+1},j_{n+1}}\neq \emptyset$ for every finite set $K$.

\end{enumerate}

Since $H\subseteq E$ and Case (1) does not hold for $E$, we find $i_0,j_0$ such that 
  $\rho(H\setminus K)\cap A_{i_0,j_0}\neq \emptyset$ for every finite set $K$.
  Since $A_{i_0,j_0}\in \I_\rho$, there is $P_0\in \cF$ which satisfies items (ii) and (iii).

Now, suppose that $i_l,j_l$ and $P_l$ have been constructed for $l\leq n$.
Since 
 $H\subseteq E$ and Case (1) does not hold for $E$, there is $k_0$ such that for every $k>k_0$ there is a finite set $K$ with $\rho(H\setminus K)\cap \bigcup\{A_{k,j}:j\in \omega\}=\emptyset$, so, by item (iii),  there is a finite set $L$ with 
 $\rho(P_n\setminus L)\cap \bigcup\{A_{k,j}:j\in \omega\}=\emptyset$.
Since  $\{A_{i,j}: i,j\in \omega\}$ is  a witness for $\rho \notin P^-_2(\Lambda)$, there are  $i_{n+1},j_{n+1}$ such that $\rho(P_n\setminus K)\cap A_{i_{n+1},j_{n+1}}\neq \emptyset$ for every finite set $K$. 
  Since $\bigcup\{A_{i_l,j_l}:l\leq n\}\in \I_\rho$, there is $P_{n+1}\in \cF$ which satisfies items (ii) and (iii).
This finishes the construction of $i_n,j_n$ and $P_n$.

Now, we notice that $\{(i_n,j_n):n\in \omega\}$ is infinite by items (ii) and (iv).
To see  that the sequence $(i_n,j_n)$ is as required in the claim, suppose for the sake of contradiction that there is $n$ and a finite set $K$ with $\rho(H\setminus K)\cap A_{i_n,j_n}=\emptyset$.
By item (i), we know that $n\geq 1$, By item (iii),  there is a finite set $L$ with $\rho(P_{n-1}\setminus L)\cap A_{i_n,j_n}=\emptyset$ which is a contradiction with item (iv).
\end{proof}

Now, the proof proceeds similarly to the proof of the previous case. However, for the sake of completeness, we present it here with minor modifications.
Let $g:\omega\to\omega$ be any function such that $g^{-1}(\{n\})$ is infinite for each $n\in \omega$.
Let $\Omega=\{m_i :i \in \omega\}$ and  $ M_n=\{m_i : i<n \} $ for each $n\in \omega$. 
Now, for each $n\in \omega$ we pick any 
\begin{equation*}
a_n \in \rho\left(E\setminus \left(\bigcup_{k<n} K_{a_k} \cup M_n\right)\right)\cap  A_{i_{g(n)},j_{g(n)}}.
\end{equation*}
Let 
$G = \{f(a_n):n\in \omega\}$. 
Since $a_n$ are pairwise distinct and $f$ is injective,
we obtain that 
$G\cap (\{(i_{n},j_{n})\}\times \omega)$ is infinite for each $n\in \omega$. 
We have two cases.

\emph{Case (a)}.
If $\{i_{n}:n\in \omega\}$ is finite, then there is $i$ such that 
$C=\{n:i_n=i\}$ is infinite.
Then, $\{j_n:n\in C\}$ is infinite, so 
$G\cap (\{(i,j_{n})\}\times \omega)$ is infinite for every $n\in C$. Consequently, $G\notin \BI$.

\emph{Case (b)}.
If $\{i_{n}:n\in \omega\}$ is infinite, then 
$G\cap (\{i_{n}\}\times\omega\times \omega)$ is infinite for each $n\in \omega$. Consequently, $G\notin \BI$.

Take any finite set $K$ and let $i$ be such that $K\subseteq M_i$.
Let $L=\{f(a_n):n<i\}$.
Then, $L$ is finite and 
for every $f(a_n)\in G\setminus L$ we have
\begin{equation*}
f(a_n)\in  f[\rho(E\setminus M_i)] \subseteq f[\rho(F\setminus K)],
\end{equation*}
so $G\setminus L \subseteq f[\rho(F\setminus K)]$.


$(\ref{thm:rho-BI-vs-P2minus-vs-FinBW:item-P2minus-vs-FinBW}, \implies)$ 
     Let $f: \Lambda \to \omega^2 +1$. We need to find some $F \in \mathcal{F}$  such that $f \restriction \rho(F)$ is $\rho$-convergent.
If there exists some $\alpha \in \omega^2 +1$ with  $f^{-1}(\alpha) \notin \I_{\rho}$, then there exists $F \in \mathcal{F}$ such that $\rho(F) \subseteq f^{-1}(\alpha) $, so $f \restriction \rho(F)$ is constant therefore $\rho$-convergent.
Now, we assume that $f^{-1}(\alpha) \in \I_{\rho}$ for each $\alpha \in \omega^2 +1$.
   Observe that if an ordinal $\alpha \in \omega^2 $ then there exists some  $i,k \in \omega $ for which $\alpha=\omega\cdot i +k$. We define a partition $(A_{i,k})_{i,k\in \omega}$ of $\Lambda$ by $A_{0,0}=f^{-1}(0)\cup  f^{-1}(\omega^2)$ and $A_{i,k}=f^{-1}(\omega\cdot i +k)$ if either $k \neq 0 $ or $ l\neq 0$. Then, $A_{i,k} \in \I_{\rho}$ for each $i,k$. 
   Since  $\rho \in P_2^-(\Lambda)$, there exists $E \in \mathcal{F}$ such that
    \begin{enumerate}
    \item $ \forall i,k \, \exists K\in [\Omega]^{<\omega} \, (\rho(E\setminus K)\cap A_{i,k}=\emptyset) $,
    \item $\exists i_0\,\forall i>i_0\, \exists K\in [\Omega]^{<\omega} \, (\rho(E\setminus K)\cap \bigcup_{k<\omega} A_{i,k}=\emptyset) $.
    \end{enumerate}
Since $\rho\in (S_2)$,  there exists $ F \in \mathcal{F}, F\subseteq E$ 
such that 
 for every $B \notin \I_{\rho}$ with $B \subseteq \rho(F)$ 
 there exists $G \in \mathcal{F}$ with $\rho(G) \subseteq B$
 such that 
\begin{equation*}
\forall K\in [\Omega]^{<\omega} \, \exists L \in [\Omega]^{<\omega} \, (\rho (G\setminus L) \subseteq \rho(F \setminus K)).
\end{equation*}
   
\emph{Case (1).}  $B = \rho(F)\cap \bigcup_{k<\omega} A_{i_0,k} \notin \I_{\rho}$
for some  $i_0\in \omega$. 
   
 Let  $ G \in \mathcal{F}$ be such that  $\rho(G) \subseteq B$ 
and 
\begin{equation*}
\forall K\in [\Omega]^{<\omega} \, \exists L \in [\Omega]^{<\omega} \, (\rho (G\setminus L) \subseteq \rho(F \setminus K)).
\end{equation*}
We will show that $f \restriction \rho(G)$ is $\rho$-convergent to $\omega\cdot (i_0+1)$. 
    Let $U$ be a neighborhood of $\omega\cdot (i_0+1)$. Then, there exists $k_0$ such that $[\omega\cdot i_0+k_0, \omega\cdot(i_0+1)]\subseteq U$. 
    For each $ k<k_0$, we take a finite set $K_k$ such that 
    $\rho(E\setminus K_k)\cap A_{i_0,k}=\emptyset$. 
Since $K =\bigcup_{k<k_0}K_k$ is finite, there exists a finite set $L$ such that 
$\rho(G\setminus L)\subseteq \rho(F\setminus K)$.
Then, 
\begin{equation*}
    \begin{split}        
\rho(G\setminus L) \cap \bigcup_{k<k_0} A_{i_0,k} 
&\subseteq
\rho\left(F\setminus K\right) \cap \bigcup_{k<k_0} A_{i_0,k} 
\subseteq
\rho\left(E\setminus K\right) \cap \bigcup_{k<k_0} A_{i_0,k} 
=\emptyset,
    \end{split}
\end{equation*}
so 
$f[\rho(G\setminus L)] \subseteq [\omega\cdot i_0+k_0,\omega\cdot (i_0+1)]\subseteq U$.
    
\emph{Case (2).} $\rho(F)\cap \bigcup_{k<\omega} A_{i,k} \in \I_{\rho}$ for every $i\in \omega$. 

For every $i>i_0$ there is a finite set $K_i$ such that 
$\rho(E\setminus K_i)\cap \bigcup_{j<\omega} A_{i,j}=\emptyset$.

Let 
\begin{equation*}
B = \bigcup_{i>i_0} \left(\rho(F)\cap \bigcup_{k<\omega} A_{i,k}\right).
\end{equation*}
Since $B\notin \I_\rho$ (as $B=\rho(F)\setminus\bigcup_{i\leq i_0}(\rho(F)\cap\bigcup_{k<\omega}A_{i,k})$) and $B\subseteq \rho(F)$, there is $G\in \cF$ such that  
$\rho(G) \subseteq B$
and 
\begin{equation*}
\forall K\in [\Omega]^{<\omega} \, \exists L \in [\Omega]^{<\omega} \, (\rho (G\setminus L) \subseteq \rho(F \setminus K)).
\end{equation*}
We will show that $f \restriction \rho(G)$ is $\rho$-convergent to $\omega^2$. 
Let $U$ be the neighborhood of $\omega^2$. Then, there exists $j_0$ such that $[\omega\cdot j_0,\omega^2]\subseteq U$. 
If $j_0<i_0$, then $\rho(G)\subseteq \bigcup_{i>i_0}  \bigcup_{k<\omega} A_{i,k}$,
so
$f[\rho(G)] \subseteq [\omega\cdot i_0,\omega^2]\subseteq [\omega\cdot j_0,\omega^2]\subseteq U$.
Now, assume that  $j_0\geq i_0$. 
Since  
$K=\bigcup_{i_0<i<j_0} K_i$
 is finite,  there exists a finite set $L$ such that 
$\rho(G\setminus L)  \subseteq \rho(F\setminus K)$.
Then, 
\begin{equation*}
    \begin{split}        
f[\rho(G\setminus L)]
&\subseteq 
f[\rho(F\setminus K)]
\subseteq
f[\rho(E\setminus K)]
\\&
\subseteq
f\left[\rho(E\setminus K)\cap \bigcup_{i_0<i<j_0} \bigcup_{k<\omega} A_{i,k}\right]
\cup
f\left[\rho(E\setminus K)\cap \bigcup_{i\geq j_0} \bigcup_{k<\omega} A_{i,k}\right]
\\&=
\emptyset\cup
f\left[\rho(E\setminus K)\cap \bigcup_{i\geq j_0} \bigcup_{k<\omega} A_{i,k}\right]
\subseteq 
[\omega\cdot j_0,\omega^2]\subseteq U.
    \end{split}
\end{equation*}


$(\ref{thm:rho-BI-vs-P2minus-vs-FinBW:item-P2minus-vs-FinBW}, \impliedby)$ 
    Take any partition $\{A_{i,j}:i,j\in \omega\}$ of $\Lambda$ such that $A_{i,j}\in \I_\rho$ for each $i,j$.
    Let $f:\Lambda\to \omega^2+1$ be a function such that  
    $f[A_{i,j}]\subseteq \{\omega\cdot i+j\}$ for every $i,j$.
Since $\omega^2+1\in \FinBW(\rho)$, there exists $F\in \cF$ such that $f\restriction\rho(F)$ is $\rho$-convergent to $p\in \omega^2+1$.
We have three cases.

\emph{Case (1).}
$p=\omega^2$.

Take any $i\in \omega$.
Then,  $U=(\omega\cdot (i+1),\omega^2]$ is an open  neighborhood of $p$, so there is a finite $K$ with $f[\rho(F\setminus K)] \subseteq U$.
Thus, $\rho(F\setminus K)\cap \bigcup_{j<\omega}A_{i,j}=\emptyset$.

\emph{Case (2).}
$p=\omega\cdot i_0$ for some $i_0\in \omega\setminus\{0\}$.

Take any $i\neq i_0-1$ and $j\in \omega$.
Then, $U=(\omega\cdot (i_0-1)+j,\omega\cdot i_0]$ is an open neighborhood of $p$, so there is a finite $K$ with $f[\rho(F\setminus K)] \subseteq U$.
Thus, $\rho(F\setminus K)\cap \bigcup_{j<\omega}A_{i,j}=\emptyset$
and 
$\rho(F\setminus K)\cap A_{i_0-1,j}=\emptyset$.

\emph{Case (3).}
$p=0$ or $p=\omega\cdot i_0+j_0$ for some $i_0\in \omega$ and $j_0\in \omega\setminus\{0\}$.

In this case, $U=\{p\}$ is an open neighborhood of $p$, so there is a finite $K$ with $f[\rho(F\setminus K)] \subseteq U$. If $p=0$ we obtain $\rho(F\setminus K) \subseteq A_{0,0}$, so $A_{0,0}\notin \I_\rho$, a contradiction. Similarly, if $p=\omega\cdot i_0+j_0$ for some $i_0\in \omega$ and $j_0\in \omega\setminus\{0\}$, we obtain $\rho(F\setminus K) \subseteq A_{i_0,j_0}$, so $A_{i_0,j_0}\notin \I_\rho$, again a contradiction.
\end{proof}

Recall that a sequentially compact space $X$ is called \emph{boring} (\cite[Definition~3.21]{MR4584767}) if there exists a finite set $F\subseteq X$ such that each one-to-one sequence in $X$ converges to some point of  $F$.

\begin{theorem}
\label{thm:Boring-Spaces-via-BI}
Let $\rho : \mathcal{F} \to [\Lambda]^{\omega}$ be partition regular.
If  $\rho \in (S_2)$, then the following conditions are equivalent.
    \begin{enumerate}
        \item $\rho_{\BI}\leq_K\rho$ and  $\rho_{\Fin^2} \not\leq_K \rho$.\label{thm:Boring-Spaces-via-BI:item-rho-BI}
        
        \item $\rho\in P_1^-(\Lambda)$ and $\rho \notin P_2^-(\Lambda)$.\label{thm:Boring-Spaces-via-BI:item-Pminus} 

        \item  $\FinBW(\rho)$ coincides with the class of all boring spaces.\label{thm:Boring-Spaces-via-BI:item-FinBW}
        \end{enumerate}
\end{theorem}

\begin{proof}
     $(\ref{thm:Boring-Spaces-via-BI:item-rho-BI}) \iff (\ref{thm:Boring-Spaces-via-BI:item-Pminus})$
It follows from 
Theorems~\ref{thm:Finite-Spaces-via-FinSqrd}
and \ref{thm:rho-BI-vs-P2minus-vs-FinBW}.


$(\ref{thm:Boring-Spaces-via-BI:item-Pminus}) \implies (\ref{thm:Boring-Spaces-via-BI:item-FinBW})$
First, we show that no unboring space is in $\FinBW(\rho)$.
Let $X$ be an unboring space.
Then, using \cite[Proposition~3.2]{MR4584767}, we can pick distinct points  $x_i \in X $ for $i\in \Lambda$ and injective sequences $(x_{i,j})_{j \in \Lambda}$  that are convergent to $x_i$, respectively. Since $\rho \notin P_2^-(\Lambda)$,  there exists a  partition $(A_{i,j})_{i,j \in \omega}$ of $\Lambda$ witnessing the lack of $P^-_2(\Lambda)$.
 Define a function $f:\Lambda \to \omega^2$ by $f(\lambda)=(i,j)$ for each  $\lambda \in A_{i,j}$ and a function $g:\Lambda \to X$ by $g(\lambda)=x_{f(\lambda)}$. Suppose for the sake of contradiction that there exists some $F\in \mathcal{F} $ such that $g \restriction \rho(F)$ is $\rho$-convergent to some $p \in X$. 
Since $\rho \notin P_2^-(\Lambda)$, 
we have two cases.

\emph{Case (1)}.
For every $l$ there is some $i_l\geq l$ such that for all finite $K$,  
\begin{equation*}
\rho(F\setminus K) \cap \bigcup_{j\in \omega} A_{i_l,j} \neq \emptyset
\end{equation*}

(in particular, $f[\rho(F\setminus K)] \cap (\{i_l\} \times \omega) \neq \emptyset$ as $\bigcup_{j\in\omega}A_{i_l,j}=f^{-1}[\{i_l\}\times\omega]$). Then we can pick $l_1,l_2$  such that $i_{l_1}\neq i_{l_2}$. Hence, $x_{i_{l_1}}\neq x_{i_{l_2}}$, so one of them has to be distinct from $p$. Without loss of generality let us assume $x_{i_{l_1}} \neq p$.  Since $X$ is Hausdorff, we can take neighborhoods $U$ of $x_{i_{l_1}}$ and $V$ of $p$ that are disjoint and assume without loss of generality that $V \cap \{x_{i_{l_1},j} : j \in \omega \} = \emptyset$. From $\rho$-convergence of $g\restriction\rho(F)$ to $p$ there is some finite $K$ such that $g[\rho(F\setminus K)] \subseteq V$. However, $f[\rho(F\setminus K)] \cap (\{i_{l_1}\} \times \omega)) \neq \emptyset$, so there is some  $\lambda_0 \in \rho(F \setminus K) $ with $ f(\lambda_0) \in \{i_{l_1}\} \times \omega$. Let  $ j \in \omega $ be such that $ f(\lambda_0)=(i_{l_1},j)$. Then, $g(\lambda_0)=x_{f(\lambda_0)}=x_{i_{l_1},j} \notin V$, but on the other hand, $g(\lambda_0) \in g[\rho(F\setminus K)] \subseteq V$, a contradiction.

\emph{Case (2)}.
Case (1) does not hold, and consequently there are $i_0,j_0$ such that 
$\rho(F\setminus K)\cap A_{i_0,j_0}\neq\emptyset$
for every finite set $K$.

\begin{claim}
    There are $i_1,j_1$ such that $(i_0,j_0)\neq (i_1,j_1)$ and 
$\rho(F\setminus K)\cap A_{i_1,j_1}\neq\emptyset$
for every finite set $K$.
\end{claim}

\begin{proof}[Proof of claim]
        Since $\rho\in (S_1)$, 
there    exist $H \in \mathcal{F}$ with $H\subseteq F$ and  
$P \in \mathcal{F}$ with $\rho(P) \subseteq \rho(H) \setminus A_{i_0,j_0}$ such that
for every finite set $K$ there exists a finite set $L$ with $\rho(P\setminus L)\subseteq \rho(H\setminus K)$.
Since 
 $H\subseteq F$ and Case (1) does not hold,  there is $k_0$ such that for every $k>k_0$ there is a finite set $K$ with $\rho(H\setminus K)\cap \bigcup\{A_{k,j}:j\in \omega\}=\emptyset$, so there is a finite set $L$ with 
 $\rho(P\setminus L)\cap \bigcup\{A_{k,j}:j\in \omega\}=\emptyset$.
Since  $\{A_{i,j}: i,j\in \omega\}$ is  a witness for $\rho \notin P^-_2(\Lambda)$, there are  $i_{1},j_{1}$ such that $\rho(P\setminus K)\cap A_{i_{1},j_{1}}\neq \emptyset$ for every finite set $K$. 
Since
$\rho(P) \subseteq \rho(H) \setminus A_{i_0,j_0}$, we obtain $(i_0,j_0)\neq (i_1,j_1)$.  
Moreover, 
$\rho(F\setminus K)\cap A_{i_1,j_1}\neq\emptyset$
for every finite set $K$. Indeed, suppose there is a finite set $K$ with 
$\rho(F\setminus K)\cap A_{i_1,j_1}=\emptyset$.
Then 
there exists a finite set $L$ with $\rho(P\setminus L)\subseteq \rho(H\setminus K)\subseteq\rho(F\setminus K)$, so 
$\rho(P\setminus K)\cap A_{i_1,j_1}=\emptyset$, a contradiction.
\end{proof}

Now, the proof proceeds similarly to the proof of the previous case. However, for the sake of completeness, we present it here with minor modifications.
Since  $x_{i_{0},j_0}\neq x_{i_{1},j_1}$,  one of them has to be distinct from $p$. Without loss of generality let us assume $x_{i_{0},j_0} \neq p$.  Take a neighborhoods $U$ of $p$ with $x_{i_0,j_0}\notin U$.
From $\rho$-convergence of $g\restriction\rho(F)$ to $p$ there is some finite $K$ such that $g[\rho(F\setminus K)] \subseteq U$. However, $\rho(F\setminus K) \cap A_{i_0,j_0} \neq \emptyset$, so there is some  $\lambda_0 \in \rho(F \setminus K) $ with $ f(\lambda_0) = (i_0,j_0)$. Then, $g(\lambda_0)=x_{f(\lambda_0)}=x_{i_0,j_0} \notin U$, but on the other hand, $g(\lambda_0) \in g[\rho(F\setminus K)] \subseteq U$, a contradiction. 

Now, we show that every boring space is in $\FinBW(\rho)$ (in the proof we will only use the property $ P^-_1(\Lambda)$). 
Let $X$ be a boring space. 
By \cite[Proposition~3.2]{MR4584767},   there are $X_i\subseteq X$ for $i\leq n$ such that $X=\bigcup_{i \leq n}X_i $, $X_i \cap X_j=\emptyset $ whenever $i \neq j$, and for each $i$ there is an $x_i\in X$ which is a limit point of each injective convergent sequence in $X_i$. Take any $f: \Lambda \to X$.

If there exists some $ x \in X $ such that $ f^{-1}[\{x\}] \notin \I_{\rho}$, there exists some $F \in \mathcal{F}$ with  $\rho(F) \subseteq f^{-1}[\{x\}]$, so  $f \restriction \rho(F)$ is $\rho$-convergent to $x$.

Now, assume  $f^{-1}[\{x\}] \in \I_{\rho}$ for every $x\in X$. Since $\rho$ is partition regular, there is $i_0\leq n$ such that $f^{-1}[X_{i_0}]\notin\I_\rho$.
Then, $\{f^{-1}[\{x\}] :x \in X_{i_0}\}$ is  a partition of $f^{-1}[X_{i_0}]$ into sets belonging to $\I_{\rho}$. Since $\rho \in P_1^-(\Lambda)$,  there is $E\in\cF$ such that $\rho(E)\subseteq f^{-1}[X_{i_0}]$ and for each $x\in X_{i_0}$ there is some finite $K_x$ with $\rho(E \setminus K_x) \cap f^{-1}[\{x\}]=\emptyset $. 

We claim that $f \restriction \rho(E)$ is $\rho$-convergent to $x_{i_0}$. 
Let $U$ be a neighborhood of $x_{i_0}$. Then $X_{i_0}\setminus U$
 has to be finite (otherwise we would obtain an injective sequence in $X_{i_0}$ not convergent to $x_{i_0}$). Then for $K=\bigcup_{x\in X_{i_0}\setminus U}K_x$ we have $\rho(E \setminus K) \cap f^{-1}[X_{i_0}\setminus U]=\emptyset $, so $f[\rho(E\setminus K)]\subseteq U$.
 

$(\ref{thm:Boring-Spaces-via-BI:item-FinBW}) \implies (\ref{thm:Boring-Spaces-via-BI:item-Pminus})$
Suppose for the sake of contradiction that $\rho\notin P^-_1(\Lambda)$ or $\rho\in P^-_2(\Lambda)$.
If $\rho \notin P_1^-(\Lambda)$ then by Theorem ~\ref{thm:Finite-Spaces-via-FinSqrd}, $ \FinBW(\rho)$ contains only finite spaces, a contradiction.    
If $\rho \in P_2^-(\Lambda)$, then by Theorem~\ref{thm:rho-BI-vs-P2minus-vs-FinBW}, $\omega^2 +1 \in \FinBW(\rho)$.
But $\omega^2 +1$ is not a boring space, a contradiction.
\end{proof}

\begin{corollary}[{\cite[Theorem~6.5 and Proposition~4.4]{MR4584767}}]
\label{cor:faiusgfuagsfdl}
    Let  $\I$ be an ideal on $\omega$. The following conditions are equivalent.
    \begin{enumerate}
        \item  $\FinBW(\I)$  coincides with the class of all boring spaces.
        \item $\BI\leq_K\I$ and $\Fin^2 \not\leq_K \I.$
        \end{enumerate}
\end{corollary}

\begin{proof}
Take $\rho=\rho_\I$ and apply  Theorem~\ref{thm:Boring-Spaces-via-BI}
and \cite[Propositions~6.5(2), 7.5(2b) and 10.2(4)]{FKK-Unified}.
\end{proof}

\begin{question}
Are there  partition regular functions 
$\eta_1$ and $\eta_2$ such that the following conditions are equivalent 
\begin{enumerate}
        \item  $\FinBW(\rho)$ coincides with the class of all boring spaces.
        \item $\eta_2\leq_K\rho$ and  $\eta_1 \not\leq_K \rho$
\end{enumerate}
for every partition regular function $\rho$?
\end{question}


\section{Brand-new class of  partition regular functions}

If for every partition regular function $\rho$ there would be an ideal $\J$  
with $\rho\approx_K\rho_\J$, then one could 
 deduce Theorem~\ref{thm:Finite-Spaces-via-FinSqrd} (Theorem~\ref{thm:Boring-Spaces-via-BI}, resp.) from Corollary \ref{cor:ksahdflkhsadflksad} (Corollary~\ref{cor:faiusgfuagsfdl}, resp.).
 However, as was mentioned in the preliminaries, the partition regular function $\FS$ satisfies $\FS\neq \rho_\J$ for any ideal $\J$, and from the proof it follows that ``$\neq$'' can be replaced by ``$\not\approx_K$''.
Below, we provide a collection of additional examples of partition regular functions of that sort (Proposition~\ref{prop:kaufland}).

\begin{definition}
\label{def:rho-brand-new}
    Let $\I$ be an  ideal on $\omega$,  $\mathcal{A}=\{A_{\alpha} : \alpha<\mathfrak{c}\}$ be an almost disjoint family on $\omega$,  $\overline{\mathcal{A}}=\{A\setminus K : A \in \mathcal{A}, K \in [\omega]^{<\omega} \}$, and $\I^+=\{ B_{\alpha} : \alpha <\mathfrak{c}\}$. 
    Let $\{P_n:n\in \omega\}$ be a partition of $\omega$ such that $P_n \in \I$ for each $n\in \omega$. 
    We define a function 
    $\rho^{(\I)} : \overline{\mathcal{A}} \to [\omega]^{\omega} $ by 
\begin{equation*}
\rho^{(\I)}(A_{\alpha}\setminus K)=B_{\alpha} \setminus \bigcup \{P_n: n<\max(K\cap A_{\alpha})\}
\end{equation*}
with the convention that $\max\emptyset =0$.
The definition of $\rho^{(\I)}$ depends on the family $\cA$, the enumeration of $\{B_\alpha:\alpha<\continuum\}$ and the partition $\{P_n:n\in \omega\}$, but for simplicity we just write $\rho^{(\I)}$.
\end{definition}

\begin{proposition}
\label{prop:brand-new-rho}
For every ideal $\I$, 
\begin{enumerate}
    \item $\rho^{(\I)}$ is partition regular,\label{prop:brand-new-rho:item-MRS}

    \item  $\I_{\rho^{(\I)}}= \I$,\label{prop:brand-new-rho:item-ideal-via-rho}

    \item $\rho^{(\I)}$ has small accretions,\label{prop:brand-new-rho:item-small-accretions}

    \item if  $\I \in P^-$, then  $\rho^{(\I)} \in P^-$ (in Proposition~\ref{prop:P-minus-versus-P-minus}, we show that in general one cannot reverse this implication),\label{prop:brand-new-rho:item-Pminus}

\item  $\rho^{(\I)} \in (S_2)$.\label{prop:brand-new-rho:item-S-two}

\end{enumerate}
\end{proposition}

\begin{proof}
(\ref{prop:brand-new-rho:item-MRS})
Since $\overline{\cA}$ is  a nonempty family of infinite subsets of $\omega$ such that $F\setminus K\in \overline{\cA}$ for every $F\in \overline{\cA}$ and a finite set $K$, we only need to show that $\rho^{(\I)}$ satisfies 
properties (M), (R) and (S) from Definition~\ref{def:partition-regular}.

(M)
    Take any two distinct elements of $\overline{\cA}$, say  $A_{\alpha}$ and $A_{\beta}$ for some 
    $\alpha,\beta $ and some finite $K_1, K_2$, such that $A_{\alpha}\setminus K_1 \subseteq A_{\beta}\setminus K_2$.  Since $\cA$ is an almost disjoint family, $\alpha = \beta$. Note that if $A_{\alpha}\setminus K_1 \subseteq A_{\alpha}\setminus K_2$ then $\max(K_1\cap A_{\alpha}) \geq \max(K_2\cap A_{\alpha})$.
     It follows that 
\begin{equation*}
B_{\alpha} \setminus \bigcup_{n<\max(K_1\cap A_{\alpha})}P_n \subseteq B_{\alpha} \setminus \bigcup_{n<\max(K_2\cap A_{\alpha})}P_n.
\end{equation*}
    
(R)  
    Take any $\alpha$ and any finite set $K$ such that $  C \cup D=B_{\alpha}\setminus \bigcup_{n<\max(K\cap A_{\alpha})}P_n$ for some $C,D \in [\omega]^{\omega}$. Since $ B_{\alpha}\notin\I$ and $P_n\in\I$ for all $n$, either $C\notin \I$ or $ D \notin \I$. Without loss of generality, assume $D \notin \I$. Then, $D=B_{\beta}$ for some $\beta$ and $\rho^{(\I)}(A_{\beta})=B_{\beta}\subseteq D. $ 
    
(S)   
     Take any $\alpha$ and any finite set $K_1$. We need to find $\beta$ and a finite set $K_2$ with $A_{\beta}\setminus K_2 \subseteq A_{\alpha}\setminus K_1$ such that for any element $x \in B_{\beta}\setminus \bigcup_{n<\max(K_2\cap A_{\beta})}P_n $ there is some finite set $L$ for which $x \notin B_{\beta}\setminus \bigcup_{n<\max((K_2\cup L)\cap A_{\beta})}P_n $. As in the proof above, we know that for the first inclusion to be true, $\alpha$ must equal $\beta$. Let $K_2=K_1$. Take any $x \in B_{\beta}\setminus \bigcup_{n<\max(K_2\cap A_{\beta})}P_n$. Since $(P_n)_n$ is a partition of $\omega$ there is some $n_0$ with $x\in P_{n_0}$. Choose $n_1 \in A_{\beta}$ with  $\max(K_2\cup\{n_0\})<n_1$ and let $L=n_1+1$. Then, $L$ is as required.  

(\ref{prop:brand-new-rho:item-ideal-via-rho}) 
Straightforward.

(\ref{prop:brand-new-rho:item-small-accretions}) Take any $\alpha$ and any finite set $L$. We will show that the set $A_{\alpha}\setminus L$ has small accretions. Take any finite set $K$ and note that 
\begin{equation*}
    \begin{split}
 &   \left(B_{\alpha}\setminus \bigcup_{n<\max(L\cap A_{\alpha})}P_n\right)
\setminus \left(B_{\alpha}\setminus \bigcup_{n<\max((L\cup K)\cap A_{\alpha})}P_n\right)
\\
&\subseteq B_\alpha\cap \bigcup_{n<\max((K\setminus L)\cap A_{\alpha})}P_n \in\I=\I_{\rho^{(\I)}}.
    \end{split}
\end{equation*}

(\ref{prop:brand-new-rho:item-Pminus}) Take any  disjoint sets $(A_n)_{n\in \omega}$ belonging to $\I_{\rho^{(\I)}}$ such that  $\bigcup_{n \in \omega} A_n \notin \I_{\rho^{(\I)}}$. Since $\I=\I_{\rho^{(\I)}}$ and $\I\in P^-$, there exists some set $ C \notin \I$ such that $C\subseteq \bigcup_{n \in \omega} A_n $ and for each $ n \in \omega \ $ there is some finite set $K_n $ such that $(C \setminus K_n) \cap A_n=\emptyset$. Let $\alpha$ be such that $ B_{\alpha}=C $. Fix any $n\in \omega $. 
Let $n_0\in \omega$ be such that 
$K_n\subseteq \bigcup_{i<n_0}P_i$.
Let $L=\{\min(A_\alpha\setminus n_0)\}$. Then,  
\begin{equation*}
 B_{\alpha}\setminus \bigcup_{i<\max(L\cap A_{\alpha})}P_i \subseteq  B_{\alpha}\setminus K_n =C\setminus K_n.
 \end{equation*}

(\ref{prop:brand-new-rho:item-S-two})
Take any $\alpha $, a finite set $M$ and $B\notin \I$ such that $ B \subseteq B_{\alpha}\setminus \bigcup_{n<\max(M\cap A_{\alpha})}P_n $. Then, $B=B_{\beta}$ for some $\beta $. Take any finite set $K$. We want to find some finite set $L$ such that 
\begin{equation*}
B_{\beta}\setminus \bigcup_{n<\max(L\cap A_{\beta})}P_n \subseteq B_{\alpha}\setminus \bigcup_{n<\max((M\cup K)\cap A_{\alpha})}P_n.
\end{equation*}
Since $B_{\beta} \subseteq B_{\alpha}$, it is enough to satisfy the condition that  $\max(L\cap A_{\beta})\geq \max((M\cup K)\cap A_{\alpha})$. It suffices to take $L=\{a\}$ for some $a \in A_{\beta}$ with $a\geq \max((M\cup K)\cap A_{\alpha}) $.
\end{proof}

For an ideal $\I$ on $\omega$, we define the ideal $\Fin \otimes \I$ on $\omega\times\omega$ by
\begin{equation*}
A\in \Fin\otimes\I \iff \{n\in \omega: \{k\in \omega:(n,k)\in A\}\notin \I\}\in \Fin.
\end{equation*}

For an ideal $\I$ on $\omega$, an almost disjoint family  $\mathcal{A}=\{A_{\alpha} : \alpha<\mathfrak{c}\}$ on $\omega$ and an enumeration  $\I^+=\{ B_{\alpha} : \alpha <\mathfrak{c}\}$, by $\rho^{(\Fin \otimes \I)}$ we  denote the function as defined above with the partition 
given by $P_n=\{n\}\times \omega$ for each $n\in \omega$.

\begin{proposition}
\label{prop:P-minus-versus-P-minus}
For any ideal $\I$ on $\omega$,  
\begin{enumerate}
    \item $\Fin \otimes \I \notin P^-$, 
    \item $\rho^{(\Fin \otimes \I)} \in P^-$.
\end{enumerate}
\end{proposition}

 \begin{proof}
(1)
It is enough to take $C_n=(\omega\setminus n) \times \omega$ for each $n\in \omega$ as a witness for the lack of the property $P^-$.

    (2)
    Take any decreasing sequence $(C_n)_{n\in\omega}$ of $\Fin \otimes \I$-positive sets such that the difference $C_n\setminus C_{n+1}$ is in $\Fin \otimes \I$ for each $n$. We need to find $\alpha$ and a finite set $L$  such that for each $n$ there is a finite set $K_n$ with 
\begin{equation*}
B_{\alpha}\setminus \bigcup_{k<\max((K_n\cup L) \cap A_{\alpha})}(\{k\} \times \omega)=B_{\alpha}\setminus (\max((K_n\cup L)\cap A_{\alpha}) \times \omega)  \subseteq C_n.
\end{equation*}
    For each $n$,  $C_n \notin \Fin \otimes \I$, so 
     there is $i_n>n$ such that 
\begin{equation*}
(C_n)_{(i_n)} = \{k:(i_n,k)\in C\} \notin \I.
\end{equation*}
    Take $L$ to be an empty set and let $\alpha$ be such that 
\begin{equation*}
B_{\alpha}=\bigcup_{n<\omega}\{i_n\}\times (C_n)_{(i_n)} \text{ \ \  and \ \   } K=\{\min(A_\alpha\setminus i_n)\}.
\end{equation*}
Then, using the fact that $C_m\subseteq C_n$ for $m>n$, we get $B_{\alpha}\setminus (\max(K_n \cap A_{\alpha}) \times \omega) \subseteq C_n$ for every $n$. 
 \end{proof}

    \begin{proposition}
    \label{prop:kaufland}
$\rho^{(\Fin\otimes\I)} \not\approx_K \rho_\J$
for any ideals $\I$ and $\J$ on $\omega$.

\end{proposition}

\begin{proof}
We have two cases: (1) $\J\not\in P^-(\omega)$, (2) $\J\in P^-(\omega)$.

\emph{Case (1).}
If $ \J \notin P^-(\omega)$ then there is no infinite space  $ X \in \FinBW(\J)$ by Corollary~\ref{cor:ksahdflkhsadflksad}.
However 
$\rho^{(\fin\otimes\I)} \in P^- (\omega\times \omega)$
and
$\rho^{(\fin\otimes\I)} \in (S_1)$, so there is some infinite space $ X \in \FinBW(\rho^{(\fin\otimes\I)})$ by Theorem~\ref{thm:Finite-Spaces-via-FinSqrd}. Therefore $\FinBW(\rho^{(\fin\otimes\I)})\nsubseteq \FinBW(\rho_\J)$, so $\rho_\J\nleq_K \rho^{(\fin\otimes\I)}$  by  \cite[Theorem~11.1 and Proposition~10.2]{FKK-Unified}.

\emph{Case (2).}
We show that  $\rho^{(\Fin\otimes\I)} \nleq_K \rho_\J$. Let us  take any function $f: \omega  \to \omega\times\omega $.
    
\emph{Case (2a).} There is some $i\in\omega$ with $f^{-1}[\{i\} \times \omega] \notin \J$. 

Then, $f[ f^{-1}[\{i\} \times \omega]] \in \Fin \otimes \I $. Therefore for any $\alpha$ and any finite $K$, it cannot be true that $ B_{\alpha}\setminus (\max((K\cup L)\cap A_{\alpha}) \times \omega \subseteq f[ f^{-1}[\{i\} \times \omega]]$ for any finite $L$, so $f$ cannot witness the Kat\v{e}tov reduction.
    
\emph{Case (2b).} For all $i\in\omega$, $f^{-1}[\{i\} \times \omega] \in \J$.

Then, $(f^{-1}[\{i\} \times \omega])$ defines a partition of $\omega$ into sets belonging to $\J$ and since $\J \in P^-(\omega)$ there exists some $F \notin \J$ such that for every $i$ there is some finite set $K_i$ with $(F\setminus K_i) \cap f^{-1}[\{i\} \times \omega]= \emptyset $. Then, $f[F\setminus K_i] \cap (\{i\} \times \omega)= \emptyset$, so $f[F] \in \Fin \otimes \I$. Therefore for any $\alpha$ and any finite set $K$,  it cannot be true that $B_{\alpha}\setminus (\max((K\cup L)\cap A_{\alpha}) \times \omega \subseteq f[F]$ for any finite set $L$, so again $f$ cannot witness the Kat\v{e}tov reduction.
\end{proof}


\section{Critical \texorpdfstring{$\rho$}{rho} for the class of all metric compact spaces}
\label{sec:kajshfdlsa}

\begin{lemma}\label{6.1}
    For every partition regular $\rho$ and every $f\colon\Lambda \to [0,1]$ there is $g\colon\Lambda \to [0,1]\cap \Q$ such that $g$ is an injection and
\begin{equation*}
\forall{F\in\F}\, \forall{p\in[0,1]} ( f\restriction \rho(F) \text{ is } \rho\text{-convergent to $p$} \iff g\restriction \rho(F) \text{ is } \rho\text{-convergent to $p$}).
\end{equation*}
\end{lemma}
\begin{proof}
Without loss of generality, we assume that $\Omega=\Lambda=\omega$.

For each $n\in\omega$ choose $g(n) \in [0,1]\cap\mathbb{Q} \setminus \{ g(i) : i<n \}$ such that $|f(n)-g(n)|<1/2^n$. We will show that $g$ is as needed. Fix $F\in\F.$

Assume that $f\restriction\rho(F)$ is $\rho$-convergent to $p$. Fix an open neighborhood $(p-r,p+r)$ of $p$. Since $f\restriction\rho(F)$ is $\rho$-convergent to $p$, there exists a finite set $K$ such that $f[\rho(F\setminus K)]\subseteq (p-r/2,p+r/2).$ Let $k$ be chosen in a way that $1/2^k<r/2$. 
Since $\rho$ satisfies (S),
there exists a finite set $K'$  such that $\rho(F\setminus K')\cap\{0,1,\dots,k\}=\emptyset$. Then, $g[\rho(F\setminus (K\cup K'))]\subseteq (p-r,p+r).$ Thus $g\restriction\rho(F)$ is $\rho$-convergent to $p$.

Similarly, one can show that if $g\restriction\rho(F)$ is $\rho$-convergent to $p$ then  $f\restriction\rho(F)$ is also $\rho$-convergent to $p$.
\end{proof}

\begin{theorem}
\label{thm:ajsgdfjagsdkf}
    Let $\rho\colon\F\to [\Lambda]^\omega$ be partition regular with $\F\subseteq[\Omega]^\omega$. If $\rho\in P^-$ and  has small accretions, then the following conditions are equivalent.
\begin{enumerate}

        \item  $\FinBW(\rho)$  coincides with the class of all compact metric spaces in the realm of metric spaces.

    \item 
$[0,1]\in\FinBW(\rho)$.

\item $\rho_\conv \nleq_K \rho$.
\end{enumerate}
 The assumption  ``$\rho\in P^-$''  is only used to prove the implication ``$(3)\implies(2)$''.
\end{theorem}

\begin{proof}
$(1)\iff(2)$
The implication ``$\implies$'' is obvious, and to show the reverse implication
it is enough to notice that every compact  metric space is a continuous image of the Cantor space, and the latter space is homeomorphic to a closed subset of $[0,1]$. 

$(2)\implies(3)$ Let us assume that $\rho_\conv\leq_K\rho$. Then, it follows from \cite[Theorem 11.1(1)]{FKK-Unified} that
$\FinBW(\rho)\subseteq\FinBW(\rho_\conv).$
    From \cite[Proposition 10.2(4)]{FKK-Unified} we also see that
$\FinBW(\conv)=\FinBW(\rho_\conv).$ 
    Since $[0,1]\notin\FinBW(\conv)$ \cite[Section 2.7]{alcantara-phd-thesis}, we obtain $[0,1]\notin\FinBW(\rho)$.

$(3)\implies(2)$
Without loss of generality, we can assume $\Omega=\Lambda=\omega$.  
    Suppose that 
    $[0,1]\notin \FinBW(\rho)$.
    Then, there is $f\colon\omega\to[0,1]$ such that
    $f\restriction \rho(F)$ is not $\rho$-convergent for any $F\in \cF$.
By Lemma \ref{6.1}, without loss of generality, we can assume that $f\colon\omega\to[0,1]\cap\mathbb{Q}$. We are going to show that $f$ is a witness for $\rho_\conv\leq_K\rho$. Let $E_1\in\F$. Since $\rho$ has small accretions, there is $F_1\subseteq E_1$, $F_1\in\F$ such that: 
$\rho(F_1)\setminus\rho(F_1\setminus K) \in \I_\rho$ for every finite $K$. 

Define $Z^{-1}=\rho(F_1)$ and observe that $\rho(F_1\setminus K) \subseteq Z^{-1} \cap \rho(E_1\setminus K)$ for every finite $K$ (as $F_1\subseteq E_1$). We begin with a construction of $x_m^n\in\omega, A^n\subseteq\omega, y^n\in[0,1]$, open sets $U^n,V^n\subseteq[0,1]$ and $Z^n\subseteq\omega$ for all $n,m\in\omega$ that will satisfy the following conditions:
\begin{enumerate}
    \item $y^n\in V^n\subseteq U^n$,
    \item $Z^n\subseteq Z^{n-1}$ and $ Z^n\notin \I_\rho$,
    \item $f[Z^n]\cap V^n =\emptyset$,
    \item $(f(x_m^n))_{m\in A^n}$ converges to $y^n$,
     \item $x_m^n\in\rho(F_1\setminus[0,m])\cap Z^{n-1})$; in particular, $\rho(F_1\setminus[0,m])\cap Z^{n}\neq  \emptyset$ for all $n\in\omega$,
     \item $y^i\neq y^j$ for $i\neq j.$
\end{enumerate}
In the $n$-th step, for every $m\in\omega$ we choose $x_m^n\in\rho(F_1\setminus[0,m])\cap Z^{n-1}$ (this is possible by (5) or, in the case of $n=0$, by $\rho(F_1\setminus [0,m])\subseteq Z^{-1}$). The sequence $f(x_m^n)_{m\in\omega}$ has a subsequence $f(x_m^n)_{m\in A^n}$ convergent to some $y^n$. Note that by (2) for every $m$ we have 
\begin{equation*}
Z^{n-1} \setminus \rho(F_1 \setminus [0,m]) \subseteq \rho(F_1) \setminus \rho(F_1 \setminus [0,m]) \in \I_\rho.
\end{equation*}
Since $\rho$ is $P^-$,
there is $G\in\F$ such that 
$\rho(G)\subseteq Z^{n-1}$ 
and for every $m$ there exists $k$ such that $\rho(G \setminus [0,k]) \subseteq \rho(F_1 \setminus [0,m]).$
However, $f[\rho(G)]$ is not $\rho$-convergent, so it is  not $\rho$-convergent to $y^n$. Hence, there exists an open interval $U^n$ containing $y^n$ such that for all $k$ we have $f[\rho(G \setminus [0,k])] \not\subseteq U^n.$ Recall that for every $m$ there exists $k$ such that $\rho(G \setminus [0,k])\subseteq \rho(F_1 \setminus [0,m])\cap Z^{n-1}$. So also for all $m$ we have $f[\rho(F_1 \setminus [0,m])\cap Z^{n-1}] \not\subseteq U^n.$
Since $\rho\in P^-$, we know that there exists $n_0$ such that the set
\begin{equation*}
\{i\in Z^{n-1}: f(i)\notin(y^n-1/n_0,y^n+1/n_0)\}\in \I_\rho^+
\end{equation*}
(if not, $\rho\in P^-$ would give us a $\rho$-sequence $\rho$-converging to $y^n$). Put $V^n=U^n\cap(y^n-1/n_0,y^n+1/n_0)$ and $Z^n = \{i\in Z^{n-1}:f(i)\notin V^n\}.$

Now, we will verify that this construction meets all required conditions. Items (1), (3), (4) and condition $Z^n\subseteq Z^{n-1}$ from (2) are clear. To show that $Z^n\notin\I_\rho$, observe that $Z^n\supseteq \{i\in Z^{n-1}: f(i)\notin(y^n-1/n_0,y^n+1/n_0)\}\in \I_\rho^+$. Thus, (2) is satisfied. Now we consider the second part of item (5) (as the first part is clear). Since $f[\rho(F_1 \setminus [0,m])\cap Z^{n-1}] \not\subseteq U^n$, there is $x\in \rho(F_1 \setminus [0,m])\cap Z^{n-1}$ such that $f(x)\notin U^n$. Since $V^n\subseteq U^n$ (by (1)), $f(x)\notin V^n$, so $x\in \rho(F_1 \setminus [0,m])\cap Z^n$. Finally, we verify (6). Assume, for the sake of contradiction, that $y^k=y^n$ for some $k<n$. Since $V^k$ is an open neighborhood of $y^k$ (by (1)) and $y^n=y^k$ is the limit of the sequence $(f(x_m^n))_{m\in A^n}$, for almost all $m\in A^n$ we have $f(x_m^n)\in V^k$. On the other hand, by (2) and (5), $x_m^n\in Z^{n-1}\subseteq Z^k$ for all $m\in A^n$, so $f[Z^k]\cap V^k\neq \emptyset$ which contradicts  (3). 

We define a set
\begin{equation*}
X= \{x_m^n:n\in\omega, m\in A^n\setminus[0,n]\}.
\end{equation*}
Let $F_2=f[X]$. Then,  
$F_2\notin \conv$ because $X$ contains elements of infinitely many sequences convergent to different values. Now, let $K_1\in\Fin$ and find $n\in\omega$ such that $K_1\subseteq[0,n].$ We let
\begin{equation*}
K_2=f[\{x_m^k:k\leq n, m\in A^k\setminus[0,k], m\leq n\}]\in \Fin.
\end{equation*}
Then
\begin{equation*}
    \begin{split}
F_2\setminus K_2
=&
f[\{x_m^k:k>n,m\in A^k\setminus[0,k]\}]
\\&
\cup f[\{x_m^k: k\leq n, m> n, m\in A^k\setminus[0,k]\}]\subseteq f[\{x_m^k:k\in\omega, m >n\}]. 
    \end{split}
\end{equation*}
Since $\{x_m^k:k\in\omega, m >n\}\subseteq \rho(F_1\setminus[0,n])$ by (5), we conclude that $F_2\setminus K_2\subseteq f[\rho(F_1\setminus[0,n])]\subseteq f[\rho(F_1\setminus K_1)]$.
\end{proof}


\begin{definition}\label{def:tau}
Let 
$\mathcal{A}=\{A_\alpha: \alpha< \continuum\}$
be an almost disjoint family and  $X$ be the family of all $x\in\left([0,1]\cap\mathbb Q\right)^{\omega\times\omega}$ such that
\begin{itemize}
    \item for every $p\in[0,1]$ there is its open neighborhood $U$ such that $x[(\omega\setminus[0,n])\times\omega]\nsubseteq U$ for all $n\in\omega$,
    \item $x$ is an injection,
    \item $x[(\omega\setminus[0,n])\times\omega]\notin \conv$ for all $n\in\omega$.
\end{itemize}
Fix an enumeration $X=\{x_\alpha: \alpha< \continuum \}$.

Now we define
$\overline{\mathcal{A}}=\{A\setminus K: A\in\mathcal A,  K \in [\omega]^{<\omega} \}$
and a map $\rho^{[0,1]}\colon\overline{\mathcal{A}}\to [[0,1]\cap\mathbb Q]^\omega$ by
\begin{equation*}
\rho^{[0,1]}(A_{\alpha}\setminus K)=x_\alpha[(\omega\setminus[0,\max(A_\alpha\cap K)])\times\omega].
\end{equation*}

\end{definition} 

\begin{proposition}
$\rho^{[0,1]}$ is a partition regular function with $\conv=\I_{\rho^{[0,1]}}$.
\end{proposition}
\begin{proof}
Since $\overline{\cA}$ is  a nonempty family of infinite subsets of $\omega$ such that $F\setminus K\in \overline{\cA}$ for every $F\in \overline{\cA}$ and a finite set $K$, we only need to show that $\rho^{[0,1]}$ satisfies 
properties (M), (R) and (S) from Definition~\ref{def:partition-regular}.

(M): 
Take  $E=A_\alpha\setminus K_1$ and $F=A_\beta\setminus K_2$ for some $\alpha,\beta$ and finite sets $K_1,K_2$ such that  $E\subseteq F$. Since $\cA$ is an almost disjoint family, $\alpha = \beta$, hence we just
 need to check whether it is true that
$[0,\max(A_\alpha\cap K_2)]\subseteq [0,\max(A_\alpha\cap K_1)]$.
Let $m=\max(A_\alpha\cap K_2)$. 
Then,  $m\in A_\alpha \cap K_2$, 
so $m\notin A_\alpha \setminus K_2=F\supseteq E=A_\alpha\setminus K_1.$ Hence $m\notin A_\alpha\setminus K_1$, and consequently $m\in A_\alpha\cap K_1$.

(S): 
Fix $E\in\overline{\mathcal A}$. Then, there is  $\alpha$ and finite set $L$  such that $E = A_\alpha\setminus L$.
Let  $F = E$.
Fix $a\in\rho^{[0,1]}(F).$ 
We need to find a finite set  $K$ such that
$a\notin x_\alpha[(\omega\setminus[0,\max(A_\alpha\cap (K\cup L))])\times \omega]$.  
Since $x_\alpha$ is an injection, we know there exist unique $i\in\omega\setminus [0,\max(A_\alpha\cap L)]$ and $j\in\omega$ such that $x_\alpha(i,j)=a.$ Then, we find $i_0\in A_\alpha$ with $i_0\geq i$, put $K =\{i_0\}$ and observe that it is as needed.

(R): 
 Fix $\alpha$, a finite set $K$ and $A,B\subseteq\mathbb{Q}\cap[0,1]$ such that $\rho^{[0,1]}(A_\alpha\setminus K) = A\cup B$. Since $\rho^{[0,1]}(A_\alpha\setminus K) = x_\alpha[(\omega\setminus[0,\max(A_\alpha\cap K)])\times\omega]\notin\conv$, we have $A\notin\conv$ or $B\notin\conv$. Without loss of generality, we assume $A\notin\conv$. 
 The set $A$ has infinitely many distinct limit points $\{p_n : n\in\omega\}$. Let $(q_{n,i})_{i\in\omega}$ for $n\in\omega$ be sequences in $A$ such that $\lim_{i\to\infty}q_{n,i}=p_n$ and $q_{n,i}\neq q_{n',i'}$ whenever $(n,i)\neq(n',i')$. Define $x:\omega\times\omega\to [0,1]\cap\Q$ by $x(i,n) = q_{n,i}$ (note the exchange of the order of indices). 
Then, $x\in X$, so  there exists $\beta$ such that $x_\beta =x$, so $\rho^{[0,1]}(A_\beta)=x_\beta[\omega\times\omega]\subseteq A$. 

($\I_{\rho^{[0,1]}}=\conv$)
From the proof of (R), it follows  that $\conv^+\subseteq\I^+_{\rho^{[0,1]}}$, and the reversed inclusion is straightforward.
\end{proof}

\begin{lemma}
\label{lem:tau}
$[0,1]\notin\FinBW(\rho^{[0,1]})$.
\end{lemma}
\begin{proof}
    Define $f\colon [0,1]\cap\mathbb Q\to[0,1]$ by $f(q)=q$. 
Take any $A_\alpha\setminus K\in\overline{\mathcal A}$
and $p\in[0,1].$ 
Since $x_\alpha\in X$, there exists some open set $U\ni p$ such that for all finite sets  $L$,
\begin{equation*}
f[\rho^{[0,1]}(A_\alpha\setminus(K\cup L))] =\rho^{[0,1]}(A_\alpha\setminus(K\cup L))=x_\alpha[(\omega\setminus[0,\max(A_\alpha\cap (K\cup L))])\times\omega] \nsubseteq U.
\end{equation*}
    Hence, $f\restriction  \rho^{[0,1]}(A_\alpha\setminus K)$ is not $\rho^{[0,1]}$-convergent to $p$.
\end{proof}

\begin{theorem}
\label{thm:oyiewrytosr}
    For  a  partition regular function $\rho$, the following conditions are equivalent.
\begin{enumerate}

        \item  $\FinBW(\rho)$  coincides with the class of all compact metric spaces in the realm of metric spaces.

    \item 
$[0,1]\in\FinBW(\rho)$.

\item $\rho^{[0,1]} \nleq_K \rho$.
\end{enumerate}
\end{theorem} 

\begin{proof}
$(1)\iff(2)$ 
Repeat the argument from  the proof of Theorem~\ref{thm:ajsgdfjagsdkf}.

$(2)\implies(3)$ Assume that $\rho^{[0,1]}\leq_K \rho.$ Then, by \cite[Theorem 11.1(1)]{FKK-Unified} we know that $\FinBW(\rho)\subseteq \FinBW(\rho^{[0,1]})$. By Lemma \ref{lem:tau}, we know that $[0,1]\notin\FinBW(\rho^{[0,1]})$,  therefore $[0,1]\notin\FinBW(\rho)$.

$(3)\implies(2)$ Without loss of generality, we assume that $\Omega=\Lambda = \omega$. Assume that $[0,1]\notin \FinBW[\rho]$.
Then, there exists $f:\omega \to [0,1]$ such that 
$f\restriction \rho(F)$ is not $\rho$-convergent
for any  $F\in\mathcal F$. Define $g$ as in Lemma \ref{6.1}. 
We claim that $g$ is a witness for 
$\rho^{[0,1]} \leq_K \rho$. Fix $F\in \F$. 
\begin{claim}
$g[\rho(F\setminus[0,n])]\notin\nolinebreak  \conv$ for every $n\in\omega$.     \end{claim}
\begin{proof}[Proof of claim]
Assume otherwise. Then, $g[\rho(F\setminus[0,n])] = \{z_j^i: j\in\omega,i<m\}$ for finitely many  convergent sequences $(z^i_j)_{j\in \omega}$, $i<m$, with distinct limits. Define 
\begin{equation*}    
\begin{split}
Z_i
&=
g^{-1}[\{z_j^i: j\in\omega\}]\cap \rho(F\setminus[0,n])
\end{split}
\end{equation*}
and observe that $\rho(F\setminus[0,n]) = \bigcup_{i<m} Z_i,$ so there exists $i<m$ and $G\in\F$ such that $\rho(G)\subseteq Z_i$ since $\rho$ is partition regular. But then, $g\restriction\rho(G)$ is convergent, so it also is $\rho$-convergent, which is a contradiction. 
\end{proof}

Using the above claim, for every $n\in\omega$ inductively find sequences $(x_k^n)_{k\in\omega}$ convergent to some $y_n$ such that the following conditions hold for every $i,j,k,n,m\in\omega$:
\begin{itemize}
\item $x_k^n\in  g[\rho(F\setminus[0,n])]$,
    \item $n\neq m \implies y_n\neq y_m$, 
\item $(n,k)\neq(i,j)\implies x_k^n\neq x_j^i$. 
\end{itemize}

Now we consider two cases.

\emph{Case 1.}
For every $p\in [0,1]$ there is a neighborhood $U$ of $p$ such that 
for any $n\in \omega$, 
\begin{equation*}
\{x_k^i: k\in\omega, i > n\} \nsubseteq U.    
\end{equation*}

Define $x:\omega\times\omega\to [0,1]\cap \Q$ by $x(n,k) = x_k^n$. Then, there exists $\alpha$ such that $x_\alpha = x.$ We fix $K_1$ and let $K_2$ be a finite set such that $\max (A_\alpha\cap K_2)\geq\max K_1$. Then,  
\begin{equation*}
    \begin{split}
\rho^{[0,1]}(A_\alpha\setminus K_2)
&=
x[(\omega\setminus [0,\max (A_\alpha\cap K_2)])\times\omega]
\\&
\subseteq x[(\omega\setminus [0,\max (K_1)])\times\omega]\subseteq g[\rho(F\setminus [0,\max K_1])]\subseteq g[\rho(F\setminus K_1)].
    \end{split}
\end{equation*}
    
\emph{Case 2.} 
 There is  $p\in [0,1]$ such that for every neighborhood $U$ of $p$ there is  $n_U\in \omega$ with
\begin{equation*}
\{x_k^i: k\in\omega, i \geq n_U\} \subseteq U.    
\end{equation*}
 
Since  $g\restriction \rho(F)$ is not $\rho$-convergent to $p$,  there exists a neighborhood $U$ of $p$ such that  $g[\rho(F\setminus[0,n])]\nsubseteq U$
for all $n\in\omega$. 
For every $n\geq  n_U$, we pick
$a_n\in\rho(F\setminus[0,n])$ such that $g(a_n)\notin U$.
    Since $\rho$ is partition regular, without loss of generality, we assume that $a_n$ are pairwise distinct. We define $x:\omega\times\omega\to [0,1]\cap\Q$ by
\begin{equation*}
x(n,k) =\begin{cases}g(a_{n+n_U}), & \text{if }k = 0,\\
    x_k^{n+n_U}, & \text{if }k\neq 0.
    \end{cases}
\end{equation*}
\begin{claim}
$x\in X$. 
\end{claim}
\begin{proof}[Proof of claim]
It is clear that 
$x[(\omega\setminus[0,n])\times\omega]\notin\conv$ and that $x$ is injective (as $g$ is injective, $g(a_{n+n_U})\notin U$ for all $n\in\omega$, while $x_k^{n+n_U}\in U$ for all $n,k\in\omega$). Let $q \in [0,1]$ and consider two cases.

\emph{Case I:} 
$q =p$.
Then,  $q\in U$ and for every $n\in\omega$, 
\begin{equation*}
g(a_{n+1+n_U})=x(n+1,0)\in x[(\omega\setminus[0,n])\times\omega]\hspace{10pt}\textnormal{and}\hspace{10pt} g(a_{n+1+n_U})\notin  U,
\end{equation*}
hence $x[(\omega\setminus[0,n])\times\omega]\nsubseteq U.$

\emph{Case II:}
 $q \neq p$. 
 There exist disjoint open neighborhoods   $V,W$ of $p$ and $q$, respectively.
 Then, for all $n\in\omega$:
\begin{equation*}
x_1^{n_V+n+1+n_U}=x(n_V+n+1,1)\in x[(\omega\setminus[0,n])\times\omega]\hspace{10pt}\textnormal{and}\hspace{10pt} x_1^{n_V+n+1+n_U}\notin W.\qedhere
\end{equation*}
\end{proof}
Since $x\in X$, there exists $\alpha$ such that $x_\alpha = x.$ Fix $K_1$ and let $K_2$ be finite such that $\max (A_\alpha\cap K_2)\geq\max K_1$. Then, 
\begin{equation*}
    \begin{split}
\rho^{[0,1]}(A_\alpha\setminus K_2)
&=
x[(\omega\setminus [0,\max (A_\alpha\cap K_2)])\times\omega]
\\&\subseteq 
x[(\omega\setminus [0,\max (K_1)])\times\omega]\subseteq g[\rho(F\setminus [0,\max K_1])]\subseteq g[\rho(F\setminus K_1)].
    \end{split}
\end{equation*}
\end{proof}

\begin{corollary}[{\cite{alcantara-phd-thesis}}]
\label{cor:skjdfhgsdf}
    Let  $\I$ be an ideal on $\omega$. The following conditions are equivalent.
    \begin{enumerate}
        \item  $\FinBW(\I)$  coincides with the class of all compact metric spaces in the realm of metric spaces.
        \item $\conv \not\leq_K \I.$
        \end{enumerate}
\end{corollary}

\begin{proof}
Take $\rho=\rho_\I$ and apply  Theorem~\ref{thm:oyiewrytosr}
and \cite[Propositions~7.5(2b) and 10.2(4)]{FKK-Unified}.
\end{proof}

\begin{theorem}
$\rho^{[0,1]}\leq_K\rho_{\conv}$
and 
$\rho_{\conv}\nleq_K\rho^{[0,1]}$.
\end{theorem}
\begin{proof}
Assume for the sake of contradiction that 
$\rho^{[0,1]}\not\leq_K\rho_{\conv}$.
Then, 
$[0,1]\in \FinBW(\rho_\conv) = \FinBW(\conv)$
by Theorem~\ref{thm:oyiewrytosr}
and \cite[Proposition~10.2]{FKK-Unified}, a contradiction with Corollary~\ref{cor:skjdfhgsdf}.

Now, we prove $\rho_{\conv}\nleq_K\rho^{[0,1]}$.
 Take any $f: \Q\cap [0,1]\to \Q\cap [0,1]$. Define
\begin{equation*}
I_n^L = \left(\frac{1}{2^{n+2}},\frac{1}{2^{n+1}}\right)\hspace{10pt}\text{and}\hspace{10pt}I_n^R= \left( 1-\frac{1}{2^{n+1}},1-\frac{1}{2^{n+2}}\right)
\end{equation*}
    for all $n\in\omega$. Now, for all $n\in\omega$, we
    \begin{enumerate}
        \item choose a sequence $Z_n^L\subseteq  I_n^L\cap\Q$ convergent to some $z_n^L\in I_n^L$, 
        \item take an infinite $Y_n^L\subseteq Z_n^L$ such that $f[Y_n^L]$ is convergent to some $y_n^L$,
        \item define $H_n^L=\{h\in Y_n^L:|f(h)-y_n^L|\leq1/n\},$
        \item take a subsequence $(y_n^L)_{n\in G^L}$ convergent to  $u^L$, where $G^L$ is infinite.
    \end{enumerate}
    Similarly, we follow the construction of $Z_n^R,z_n^R,Y_n^R,y_n^R,H_n^R,G^R$ and $u^R$.
    
For $S\in \{L,R\}$, we enumerate
\begin{equation*}
G^S = \{g_n^S:n\in\omega\}
\hspace{10pt} \text{and} \hspace{10pt}
H_n^S 
=\{h_i^{n,S}:i\in\omega\}
.
\end{equation*}

    Define $x:\omega\times\omega\to[0,1]\cap\Q$ by 
\begin{equation*}
x(2n,i)=h_i^{g_n^L,L}\hspace{10pt}\mathrm{and}\hspace{10pt}x(2n+1,i)=h_i^{g_n^R,R}.
\end{equation*}

    Notice the following properties of the sequence $x$. 
    \begin{enumerate}
        \item $x$ is an injection since $H_n^L\subseteq Y_n^L\subseteq Z_n^L\subseteq I_n^L$ (similarly when replacing $L$ with $R$ ) and the intervals $I_n^L$ and $I_n^R$ are pairwise disjoint.
        \item $x[(\omega\setminus[0,n])\times\omega]\notin\conv$ since for all $k>n$ we can find a sequence in $x[(\omega\setminus[0,n])\times\omega]$ convergent to $z_{g_{k}^L}^L$. 
        \item If $p = 0$ then $U=(-1/2,1/2)$ is a neighborhood of $p$ and for all $n\in\omega$ we have $x[(\omega\setminus[0,n])\times\omega]\nsubseteq U$ since $h_0^{g_n^R,R}\in x[(\omega\setminus[0,n])\times\omega]$, however $h_0^{g_n^R,R}\notin U$. 
        
        \item If  $p\in(0,1]$ then there exists an open neighborhood $U$ of $p$ such that
there exists $n\in\omega$ with $I^L_{g_k^L}\cap U = \emptyset$ for all $k>n$.
So $x[(\omega\setminus[0,n])\times\omega]\nsubseteq U$ because $h_0^{g_k^L,L}\in x[(\omega\setminus[0,n])\times\omega]$, however $h_0^{g_k^L,L}\notin U.$
    \end{enumerate}
    
    Therefore $x\in X$, so there exists $\alpha$ such that $x=x_\alpha.$

Take any $B\subseteq \Q\cap [0,1]$
such that 
for every finite set $F$ there exists a finite set  $K$
such that 
$B\setminus K=\rho_{\conv}(B\setminus K) \subseteq f[\rho^{[0,1]}(A_\alpha\setminus F)]$. Once we show that $B\in \conv$, the proof will be finished. 

There exists a finite set $K$ such that
\begin{equation*}
B\setminus K \subseteq f[\rho^{[0,1]}(A_\alpha\setminus \emptyset)]= f\left[\bigcup_{n\in\omega}H_{g_n^L}^L\right]\cup f\left[\bigcup_{n\in\omega}H_{g_n^R}^R\right].
\end{equation*}
Let
\begin{equation*}
B^L = (B\setminus K)\cap f\left[\bigcup_{n\in\omega}H_{g_n^L}^L\right]\hspace{10pt}\mathrm{and}\hspace{10pt}B^R=(B\setminus K)\cap f\left[\bigcup_{n\in\omega}H_{g_n^R}^R\right].
\end{equation*}
We are going to show that $B^L$ is convergent to  $u^L$ and similarly one can show that $B^R$ converges to $u^R$, so $B=B^L\cup B^R\in\conv$.

Let $\varepsilon>0$. Then, there exists $k\in\omega$ such that for all $n>k$
\begin{equation*}
|y^L_{g_n^L}-u^L|<\frac{\varepsilon}{2}\hspace{10pt} \mathrm{and}\hspace{10pt} \frac{1}{g_n^L}<\frac{\varepsilon}{2}.
\end{equation*}
Let $F$ be a finite set such that $\max(F\cap A_\alpha)>2k$. Then, there exists $M\in[[0,1]\cap\Q]^{<\omega}$ such that $B\setminus M\subseteq f[\rho^{[0,1]}(A_\alpha\setminus F)]$. We are going to show that  $|b-u^L|<\varepsilon$
for all $b\in B^L\setminus M$.

Fix $b\in B^L\setminus M\subseteq f[\rho^{[0,1]}(A_\alpha\setminus F)]\cap f[\bigcup_{n\in\omega} H_{g_n^L}^L]$. There exist $m,i\in\omega$ such that $b=f(x_\alpha(2m,i))$ and $2m> 2k.$ Since $x_\alpha (2m,i)= h_i^{g_m^L,L} \in H_{g_m^L}^L$, we obtain that
\begin{equation*}
|b-u^L|\leq |b-y_{g_m^L}^L|+|y_{g_m^L}^L-u^L|<\frac{1}{g_m^L}+\frac{\varepsilon}{2}<\frac{\varepsilon}{2}+\frac{\varepsilon}{2}=\varepsilon.\qedhere
\end{equation*}
\end{proof}

\section{Ideal version of Mazurkiewicz theorem}

Mazurkiewicz's theorem~\cite{Mazurkiewicz1932} states that every  uniformly bounded sequence of continuous functions has a uniformly convergent subsequence once restricted to some perfect set.
The question whether the subsequence can be enumerated by $\I$-positive set for a given ideal $\I$, was investigated in \cite{MR2961261} where the authors proved that  the answer is positive for $F_\sigma$-ideals.
The following theorem shows that this result can be extended to every ideal with the $\FinBW$ property which in turn means that the ideal $\conv$ is critical for the ideal version of Mazurkiewicz's theorem.

\begin{theorem}
    Let $\I$ be an ideal on $\omega$. The following conditions are equivalent.
\begin{enumerate}
    \item $\I$ has the $\FinBW$ property.
    \item For every sequence $( f_n )_{n\in\omega}$ of uniformly bounded continuous real-valued functions defined on $\R$ there exists $A\in\I^+$ and a nonempty perfect set $P\subseteq\R$ such that the subsequence $( f_n \restriction P )_{n\in A}$ is uniformly convergent.
\end{enumerate}
\end{theorem}

\begin{proof}
$(2)\implies(1)$  Assume that $\I$ does not have the $\FinBW$ property. Then, there is a sequence $( x_n )_{n\in\omega}$ of real numbers from the interval $[0,1]$ such that for every $A\in\I^+$ the subsequence $( x_n )_{n\in A}$ is not convergent. For each $n\in\omega$, let $f_n\colon\R\to\R$ be the  constant function with the value  $x_n$. Fix any $A\in\I^+$ and a nonempty perfect set $P\subseteq\R$. Assume, for the sake of contradiction, that the subsequence $(f_n \restriction P )_{n\in A}$ is uniformly convergent. Let $x\in P$. Then, $( f_n(x) )_{n\in A}$ should converge, but $f_n(x)=x_n$ for every $n\in A$ and $( x_n )_{n\in A}$ is not convergent, a contradiction.

$(1)\implies(2)$ Fix $M>0$ such that $|f_n(x)|\leq M$ for every $x\in\R$ and $n\in\omega$. Let $\{I_s:s\in2^{<\omega}\}$ be a family of nonempty open subintervals of $(-M-1,M+1)$ satisfying:
\begin{itemize}
    \item[(C1)] $I_\emptyset = (-M-1,M+1)$,
    \item[(C2)]  $\{I_s:s\in2^{n}\}$ is an open cover of $(-M-1,M+1)$ for every $n\in\omega$,
    \item[(C3)] for every $\varepsilon>0$ there exists $n\in\omega$ such that for every $s\in2^n$, $\operatorname{diam}(I_s)<\varepsilon$,
    \item[(C4)]  $I_s=I_{s^\frown0}\cup I_{s^\frown1}$ for every $s\in2^{<\omega}$.
\end{itemize}
For each $m,n\in\omega$ and $x\in\R$, let $p(m,n,x)$ be the least, with respect to the lexicographic order, $s\in2^n$ such that $f_m(x) \in I_s$.

Let $\Q = \{ q_n : n \in \omega \}$ be an enumeration of the rationals such that $q_i\neq q_j$
for all distinct $i,j\in\omega$. For $t_0, \ldots, t_{n-1} \in 2^n$, define
\begin{equation*}
    B(t_0, \ldots, t_{n-1}) = \{
        m \in \omega : p(m, n, q_i) = t_i \text{ for every } i\in n
    \},
\end{equation*}
and set
\begin{equation*}
    \mathcal{B}^n = \{
        B(t_0, \ldots, t_{n-1}) : t_0, \ldots, t_{n-1} \in 2^n
    \}.
\end{equation*}
Then, $\mathcal{B}^n$ is a finite partition of $\omega$, and $\mathcal{B}^{n+1}$ refines $\mathcal{B}^n$. Thus, $\bigcup_{n\in\omega}\mathcal{B}^n$ is a finite-branching tree under inclusion.

Since $\I$ has the $\FinBW$ property, we can use
\cite[Proposition 2.8]{Ramsey}
to obtain  a sequence $( Z_n )_{n \in \omega}$ and a set $Z \in \I^+$ such that for every $n \in \omega$
\begin{itemize}
    \item $Z_n \in \mathcal{B}^n$,
    \item $Z_n \supseteq Z_{n+1}$,
    \item $Z \setminus Z_n \in \Fin$.
\end{itemize}

Let $( x_i )_{i\in\omega}$ be a sequence of pairwise distinct elements such that 
\begin{equation*}
Z\cap \bigcap_{n\in\omega}Z_n\subseteq\{ x_i : i \in \omega \} \subseteq Z
\end{equation*}
and $x_i\in Z_i$ for each $i\in\omega$. Define for each $n\in\omega$,
\begin{equation*}
    A_n =
        \left(\{ x_n \} \cup
        \left( 
            Z \setminus Z_{n+1} 
        \right)\right) \setminus
        \bigcup_{i < n} A_i. 
\end{equation*}
Then, $\bigcup_{n\in\omega}A_n=Z$ and $A_n\cap A_m=\emptyset$ whenever $n\neq m$. Moreover, $A_n\in\Fin$ and $A_n\subseteq Z\setminus Z_n$ for all $n\in\omega$.

For each $n \in \omega$, let $z(n, 0), \ldots, z(n, n-1) \in 2^n$ be such that 
\begin{equation*}
    Z_n = B\left(z(n, 0), \ldots, z(n, n-1)\right)
.
\end{equation*}
Define $f:\mathbb{Q}\to\mathbb{R}$ by $\{f(q_i)\}=\bigcap_{n>i} \operatorname{cl}(I_{z(n,i)})$. Note that $f$ is well-defined since $\lim_{n\to\infty}\operatorname{diam}(\operatorname{cl}(I_{z(n,i)}))=0$ (by (C3), (C4) and the fact that $z(n,i)\in 2^n$ for all $n>i$) and $I_{z(n,i)}\supseteq I_{z(n+1,i)}$ for all $n>i$ (by $B\left(z(n, 0), \ldots, z(n, n-1)\right)=Z_n \supseteq Z_{n+1}=B\left(z(n+1, 0), \ldots, z(n+1, n)\right)$). Also note that $( f_m\restriction \mathbb{Q})_{m\in Z}$ is pointwise convergent to $f$. Indeed, if $U$ is a neighborhood of some $f(q_i)$ then we can find $n>i$ such that $\operatorname{cl}(I_{z(n,i)})\subseteq U$. Then for all $m\in Z\cap Z_n$, we have $f_m(q_i)\in \operatorname{cl}(I_{z(n,i)})\subseteq U$. Since $Z\cap Z_n=Z\setminus (Z\setminus Z_n)$ and $Z\setminus Z_n$ is finite, $( f_m(q_i))_{m\in Z}$ converges to $f(q_i)$.

For each $\delta \in \{0, 1\}$, we will construct
\begin{itemize}
    \item a family of finite sequences $\{ a^\delta(s) \in 2^{<\omega} : s \in 2^{< \omega} \}$,
    \item a family of nonempty open intervals $\{ U_s^\delta : s \in 2^{< \omega} \}$,
    \item a family $\{ D_s^\delta\subseteq\mathbb{Q} : s \in 2^{< \omega} \}$,
    \item a family of rationals $\{ q_s^\delta \in \mathbb{Q} : s \in 2^{< \omega} \}$,
    \item a sequence $( k_n^\delta )_{ n \in \omega }$ in $\omega$.
\end{itemize}
Additionally, for each $ n \in \omega $, we define
\begin{equation*}
    G_n^0 = A_{k_n^0} \cup \ldots \cup A_{k_n^1 - 1} \quad \text{and} \quad 
    G_n^1 = A_{k_n^1} \cup \ldots \cup A_{k_{n+1}^0 - 1}.
\end{equation*}
These objects will be chosen so that the following conditions hold for all $\delta\in\{0,1\}$ and $s\in 2^{<\omega}$ (where $\operatorname{lh}(s)$ means the length of a finite sequence $s$):
\begin{enumerate}
    \item[(P1)] $\operatorname{lh}(a^\delta(s))=\operatorname{lh}(s)+1$, 
    \item[(P2)]  
$U_{s^\frown0}^\delta, U_{s^\frown1}^\delta \subseteq U_s^\delta$ and 
$\operatorname{cl}(U_{s^\frown0}^\delta) \cap \operatorname{cl}(U_{s^\frown1}^\delta) = \emptyset$,
    
    \item[(P3)] 
    $\operatorname{diam}\left( U_{s}^\delta \right)\leq 1/2^{\operatorname{lh}(s)}$,

    \item[(P4)] $D_{s^\frown0}^\delta, D_{s^\frown1}^\delta \subseteq D_s^\delta$,
    
    \item[(P5)]  $f[D^\delta_{s^\frown j}]\subseteq I_{a^\delta(s)}$ for each $j\in\{0,1\}$,
    
    \item[(P6)] $D^\delta_{s}\subseteq U^\delta_{s}\subseteq\operatorname{cl}(D^\delta_{s})$,
    
    \item[(P7)] 
    $f_i\left[ \operatorname{cl} \left( U_{s^\frown j}^\delta \right) \right] \subseteq 
     I_{a^\delta(s\restriction(\operatorname{lh}(s)-1))}
$
    for each  $j\in \{0, 1\}$, $i \in G_{\operatorname{lh}(s)}^\delta$ unless $s=\emptyset$,

    \item[(P8)]  $q^\delta_{s}\in D^\delta_s$,
    \item[(P9)] $k^0_0=0$ and $k_n^0 < k_n^1<k_{n+1}^0$ for each $n \in \omega$,
    \item[(P10)]  
    $f_i(q^\delta_{s^\frown j})\in I_{a^\delta(s)}$
for each     $j\in\{0,1\}$ and $i\in A_m$ for some $m\geq k_{\operatorname{lh}(s)+1}^\delta$. 
\end{enumerate}

Set $U^0_\emptyset=U^1_\emptyset=(0,1)$ and $D^0_\emptyset=D^1_\emptyset=(0,1)\cap\mathbb{Q}$. Pick any $q^0_\emptyset,q^1_\emptyset\in(0,1)\cap\mathbb{Q}$. Observe that conditions (P3), (P6) and (P8) are met. At this point there is nothing to check in (P2), (P4), (P5), (P7) and (P10) (as $U^\delta_{(j)}$, $D^\delta_{(j)}$ and $q^\delta_{(j)}$ for $\delta,j\in\{0,1\}$ are not defined yet). Put $k_0^0 = 0$ and $k_0^1 = 1$ (in particular, $k_0^0<k_0^1$, so (P9) is satisfied). The sequences $a^\delta(\emptyset)$ will be defined in the next step, so we do not have to check (P1). This ends the initial step of our construction.

Suppose now that $n\in\omega$ and that we have already defined (accordingly to (P1)-(P10)) the following objects
\begin{itemize}
    \item $a^\delta(t)$ for all $t \in 2^{<n}$ and $\delta\in\{0,1\}$,
    \item $U_s^\delta, D_s^\delta$ and $q_s^\delta$ for all $s \in 2^{\leq n}$ and $\delta\in\{0,1\}$,
    \item $k^0_i$ and $k^1_i$ for all $i\leq n$.
\end{itemize}
Note that at this point $G^0_{n}$ is defined. Let
\begin{equation*}
V^0_{s}=U^0_s\cap\bigcap_{i\in G^0_{n}}f_i^{-1}[I_{a^0(s\restriction (n-1))}]
\end{equation*}
for all $s\in 2^n$ (in the case of $n=0$ put $a^0(s\restriction (n-1))=\emptyset$). Observe that $q^0_s\in V^0_{s}$, as $q^0_{s}\in D^0_s\subseteq U^0_s$ (by (P6) and (P8)) and $f_i(q^0_{s})\in I_{a^0(s\restriction (n-1))}$ for all $i\in G^0_{n}$ (by $G^0_{n}\subseteq \bigcup_{m\geq k_{n}^0}A_m$ and item (P10); in the case of $n=0$ the condition $f_i(q^0_{s})\in I_{a^0(s\restriction (n-1))}=I_\emptyset$ is trivially satisfied). In particular, $V^0_{s}\neq\emptyset$. Moreover, by the fact that $G^0_{n}$ is finite and by the continuity of all $f_i$'s, the set $V^0_{s}$ is open. 

Observe that
\begin{equation*}
D^0_s=D^0_s\cap f^{-1}[I_{a^0(s\restriction(n-1))}]=(D^0_s\cap f^{-1}[I_{a^0(s\restriction(n-1))^\frown 0}])\cup(D^0_s\cap f^{-1}[I_{a^0(s\restriction(n-1))^\frown 1}])
\end{equation*}
(by (P5) together with (C4); in the case of $n=0$ instead of (P5) use the fact that $D^0_\emptyset\subseteq \mathbb{Q}=f^{-1}[I_\emptyset]$). Since $D^0_s$ is dense in $V^0_{s}$ (by $V^0_{s}\subseteq U^0_s$ together with (P6)), we can find a nonempty open interval $W^0_{s}\subseteq V^0_{s}$ and a sequence $a^0(s)$ (equal either to $a^0(s\restriction(n-1))^\frown 0$ or to $a^0(s\restriction(n-1))^\frown 1$) such that $D^0_s\cap f^{-1}[I_{a^0(s)}]$ is dense in $W^0_{s}$. Note that (P1) is satisfied as $\operatorname{lh}(a^0(s))=\operatorname{lh}(a^0(s\restriction(n-1)))+1=n+1=\operatorname{lh}(s)+1$ (in the case of $n=0$ we get $\operatorname{lh}(a^0(\emptyset))=1$).

Find any nonempty open intervals $U^0_{s^\frown j}$, for all $s\in 2^n$ and $j\in\{0,1\}$, satisfying $\operatorname{cl}(U^0_{s^\frown j})\subseteq W^0_{s}$ , (P2) and (P3). Also, define $D^0_{s^\frown j}=U^0_{s^\frown j}\cap D^0_s\cap f^{-1}[I_{a^0(s)}]$. Note that (P4) and (P5) are met. Since $D^0_s\cap f^{-1}[I_{a^0(s)}]$ is dense in $W^0_{s}$ and $\operatorname{cl}(U^0_{s^\frown j})\subseteq W^0_{s}$, (P6) is also satisfied. Since $\operatorname{cl}(U^0_{s^\frown j})\subseteq W^0_{s}\subseteq V^0_{s}$, we get (P7).

Pick any $q^0_{s^\frown j}\in D^0_{s^\frown j}$ for all $s\in 2^{n}$ and $j\in\{0,1\}$. Then, $f(q^0_{s^\frown j})\in I_{a^0(s)}$ (by (P5)). Since $( f_i\restriction \mathbb{Q})_{i\in Z}$ is pointwise convergent to $f$ and each $I_{a^0(s)}$ is open, there is $k^0_{n+1}>k^1_{n}$ such that $f_i(q^0_{s^\frown j})\in I_{a^0(s)}$ for all $i\in \bigcup_{m\geq k_{n+1}^0}A_m$, $s\in 2^{n}$ and $j\in\{0,1\}$. Observe that (P8), (P9) and (P10) are met.

Note that at this point $G^1_{n}$ is already defined. Repeat  the above construction for $\delta=1$ instead of $\delta=0$ to get $a^1(s), U^1_{s^\frown j}, D^1_{s^\frown j}$ and $q^1_{s^\frown j}$ for all $j\in\{0,1\}$ and $s\in 2^{n}$, as well as $k^1_{n+1}$. Hence, we have defined 
\begin{itemize}
    \item $a^\delta(t)$ for all $t \in 2^{n}$ and $\delta\in\{0,1\}$,
    \item $U_s^\delta, D_s^\delta$ and $q_s^\delta$ for all $s \in 2^{n+1}$ and $\delta\in\{0,1\}$,
    \item $k^0_{n+1}$ and $k^1_{n+1}$.
\end{itemize}
This finishes the construction.

\medskip

Observe that 
\begin{equation*}
Z=\bigcup_{n\in\omega}A_n=\left(\bigcup_{n\in\omega}G^0_n\right)\cup\left(\bigcup_{n\in\omega}G^1_n\right).
\end{equation*}
Since $Z\in\I^+$, there is $\delta\in\{0,1\}$ such that $\bigcup_{n\in\omega}G^\delta_n\in\I^+$. 
Let 
\begin{equation*}
A=\bigcup_{n\in\omega}G^\delta_n
\quad \text{and} \quad
P=\bigcap_{n\in\omega}\bigcup_{s\in 2^n}\operatorname{cl}(U_{s}^\delta).
\end{equation*}
Thanks to (P2) and (P3), $P$ is a perfect set and for each $x\in P$ there is a unique $\alpha_x\in 2^\omega$ such that $\{x\}=\bigcap_{n\in\omega}\operatorname{cl}(U^\delta_{\alpha_x\restriction n})$. Let $g\colon P\to\R$ be given by $g(x)=\limsup_{n\to\infty} f(q^\delta_{\alpha_x\restriction n})$ for all $x\in P$.

We will show that $( f_n \restriction P )_{n\in A}$ is uniformly convergent to $g$. Fix any $\varepsilon>0$ and take $N\in\omega$ such that $\operatorname{diam}(I_{a^\delta(s)})<\varepsilon$ for all $s\in 2^N$ (such $N$ exists thanks to conditions (C3) and (P1)). 

Fix any $x\in P$ and $i\in A\setminus\bigcup_{n\leq N}G_{n}^\delta$. Since $\bigcup_{n\leq N}G_{n}^\delta$ is a finite set, to finish the proof we only need to show that $|f_i(x)-g(x)|<\varepsilon$. 

As $i\in A\setminus\bigcup_{n\leq N}G_{n}^\delta$, we have $i\in G^\delta_{n+1}$ for some $n\geq N$. Then, we have $x\in\operatorname{cl}(U^\delta_{\alpha_x\restriction (n+2)})$, so by (P7) we get $f_i(x)\in I_{a^\delta(\alpha_x\restriction n)}$. Moreover, by conditions (P4) and (P8) we obtain $q^\delta_{\alpha_x\restriction (n+1+k)}\in D^\delta_{\alpha_x\restriction (n+1)}$ for all $k\in\omega$. Thus, $g(x)\in \operatorname{cl}(I_{a^\delta(\alpha_x\restriction n)})$ by (P5). Since $n\geq N$, 
$\operatorname{diam}(\operatorname{cl}(I_{a^\delta(\alpha_x\restriction n)}))=\operatorname{diam}(I_{a^\delta(\alpha_x\restriction n)})\leq \operatorname{diam}(I_{a^\delta(\alpha_x\restriction N)})<\varepsilon
$
(by (C4) and the choice of $N$). Finally, we conclude that
$|f_i(x)-g(x)|\leq\operatorname{diam}(\operatorname{cl}(I_{a^\delta(\alpha_x\restriction n)}))<\varepsilon$.
\end{proof}


\bibliographystyle{amsplain}
\bibliography{references}

\end{document}